\documentclass{gOMS2e3}

\usepackage{epstopdf}
\usepackage{subfigure}

\theoremstyle{plain}
\newtheorem{theorem}{Theorem}[section]
\newtheorem{corollary}[theorem]{Corollary}
\newtheorem{lemma}[theorem]{Lemma}
\newtheorem{proposition}[theorem]{Proposition}
\newtheorem{assumption}[theorem]{Assumption}
\newtheorem{example}[theorem]{Example}

\theoremstyle{definition}
\newtheorem{definition}{Definition}

\theoremstyle{remark}
\newtheorem{remark}{Remark}

\usepackage{graphicx}
\usepackage{multirow}
\usepackage{amssymb,amsfonts}  
\usepackage{latexsym,epsfig}
\usepackage{graphicx} 
\usepackage{dsfont}
\usepackage{lscape}
\usepackage{color}
\usepackage{float}
\usepackage{enumerate}
\usepackage[usenames,dvipsnames]{pstricks}
\usepackage{epsfig}
\usepackage{pst-grad} 
\usepackage{pst-plot} 
\usepackage{mathrsfs}
\usepackage{extarrows}
\usepackage[percent]{overpic}
\usepackage{epstopdf}
\usepackage[misc]{ifsym}
\usepackage{lipsum}
\usepackage{graphicx}
\usepackage{epstopdf}
\usepackage{algorithm}
\usepackage{algorithmic}
\usepackage[english]{babel}
\usepackage{array}
\usepackage{anyfontsize}

\usepackage{epstopdf}
\usepackage[caption=false]{subfig}

\newcolumntype{L}[1]{>{\raggedright\let\newline\\\arraybackslash\hspace{0pt}}m{#1}}
\newcolumntype{C}[1]{>{\centering\let\newline\\\arraybackslash\hspace{0pt}}m{#1}}
\newcolumntype{R}[1]{>{\raggedleft\let\newline\\\arraybackslash\hspace{0pt}}m{#1}}

\def\cu#1{{\color{black}#1}}
\def\cub#1{{\color{black}#1}}

\def\ag#1{{\color{black}#1}}
\def\ga#1{{\color{black}#1}}

\DeclareMathOperator*{\argmax}{arg\,max}
\DeclareMathOperator{\spn}{span}

\begin{document}
	
	
	
	\title{A Dual Approach for Optimal Algorithms in Distributed Optimization over Networks}
	
	\author{
		\name{C\'esar A. Uribe\textsuperscript{a}$^\ast$\thanks{$^\ast$Corresponding author. Email: cauribe@mit.edu}, Soomin Lee\textsuperscript{b}, Alexander Gasnikov\textsuperscript{c}, and Angelia Nedi\'{c}\textsuperscript{d}}
		\affil{\textsuperscript{a}Laboratory for Information and Decision Systems, and the Institute for Data, Systems, and Society, Massachusetts Institute of Technology,
			Cambridge, Massachusetts, USA; \textsuperscript{b}Yahoo! Research,
			Sunnyvale, California, USA; \textsuperscript{c}Moscow Institute of Physics and Technology, and Institute for Information Transmission, Moscow Oblast, Russia; \textsuperscript{d}School of Electrical, Computer and Energy Engineering, Arizona State University,
			Tempe, Arizona, USA}
		\received{}
	}
	
	\maketitle
	
	\begin{abstract}
		We study dual-based algorithms for distributed convex optimization problems over networks, where the objective is to minimize a sum $\sum_{i=1}^{m}f_i(z)$ of functions over in a network. We provide complexity bounds for four different cases, namely: each function $f_i$ is strongly convex and smooth, each function is either strongly convex or smooth, and when it is convex but neither strongly convex nor smooth. Our approach is based on the dual of an appropriately formulated primal problem, which includes a graph that models the communication restrictions. We propose distributed algorithms that achieve the same optimal rates as their centralized counterparts {(up to constant and logarithmic factors)}, with an additional {optimal} cost related to the spectral properties of the network. Initially, we focus on functions for which we can explicitly minimize its Legendre–Fenchel conjugate, i.e., admissible or dual friendly functions. Then, we study distributed optimization algorithms for non-dual friendly functions, as well as a method to improve the dependency on the parameters of the functions involved. Numerical analysis of the proposed algorithms is also provided.
	\end{abstract}
	
	
	\begin{classcode}90C25 ;  90C30 ; 90C60 ; 90C35\end{classcode}
	
	\section{Introduction}
	
	This paper studies the design of efficient distributed algorithms for the solution of convex optimization problems over networks. The study of distributed algorithms can be traced back to classic papers from the 70s and 80s~\cite{bor82,tsi84,deg74}. The adoption of distributed optimization algorithms on several fronts of applied and theoretical machine learning, robotics, and resource allocation has increased the attention on such methods in recent years \cite{xia06,rab04,kon15,kra13,ned17e}. The particular flexibilities induced by the distributed setup make them suitable for large-scale learning problems involving large quantities of data  \cite{bot10,boy11,aba16,ned16w,ned15}. 

We consider the following optimization problem
\begin{align}\label{main_problem}
\min_{z \in \mathbb{R}^n} \sum\limits_{i=1}^{m}f_i(z),
\end{align}
where the each \cu{$f_i: \mathbb{R}^n \to \mathbb{R} \cup \{+\infty\}$} is a closed convex function known by an agent $i$ only, that represents a node in an arbitrary communication network.
Problem~\eqref{main_problem} is to be solved in a distributed manner by repeated interactions of a set of agents over a static network. 

Initial algorithms for distributed optimization, such as distributed subgradient methods, were shown successful for solving optimization problems in a distributed manner over networks \cite{ned09,ned09b,ram10,ned13}. Nevertheless, these algorithms are particularly slow compared with their centralized counterparts. Recently, distributed methods that achieve linear convergence rates for minimizing a sum of strongly convex and smooth (network) objective functions have been proposed. \cu{For example, \cite{beck2014optimal} studies a distributed resource allocation problem but there the objective functions are decoupled, and the agents couple through the links in a different manner (sharing resources). The agents' coupling in the problem we consider is quite different from that in~\cite{beck2014optimal}, and consequently, our methods and their analysis require more information about the graph \cub{structure. Also,} unlike~\cite{NECOARA2015209}, we explore the explicit dependency on the graph topology.} One can identify three main approaches to the study of distributed algorithms. 

\cub{
\begin{itemize}
		\item In~\cite{ned17}, a new method was proposed where it was shown that $O((m^2 + \sqrt{L/\mu}m)\log \varepsilon^{-1})$ iterations are required to find an $\varepsilon$ solution to the optimization problem when the function is $\mu$-strongly convex and $L$-smooth, and $m$ is the number of nodes in the network. In \cite{jak17}, a unifying approach was proposed, that recovers rate results from several existing algorithms such as those in \cite{shi15,qu17}. This newly proposed general method is able to recover existing rates and achieves an $\varepsilon$ precision in $O(\sqrt{L/ (\mu \lambda_2)} \log \varepsilon^{-1})$ iterations, where $\lambda_2$ is the second largest eigenvalue of the interaction matrix. These results require some minimal information about the topology of the network and provide explicit statements about the dependency of the convergence rate on the problem parameters. Specifically, polynomial scalability is shown with the network parameter for particular choices of small enough step-sizes and even uncoordinated step-sizes are allowed \cite{ned17r}. One particular advantage of this approach is that can handle time-varying and directed graphs. Nevertheless, optimal dependencies on the problem parameters and tight convergence rate bounds are far less understood. 
	\item Secondly,  in~\cite{sun17}, a new analysis technique for the convergence rate of distributed optimization algorithms via a semidefinite programming characterization was proposed. This approach provides an innovative procedure to numerically certify worst-case rates of a plethora of distributed algorithms, which can be useful to fine-tune parameters in existing algorithms based on feasibility conditions of a semidefinite program.
		\item A third approach was recently introduced in~\cite{sca17}, where the first optimal algorithm for distributed optimization problems was proposed. This new method achieves an $\varepsilon$ precision in $O(\sqrt{L/\mu}(1 + \tau / \sqrt{\gamma}) \log \varepsilon^{-1})$ iterations for \mbox{$\mu$-strongly} convex and $L$-smooth problems, where $\tau$ is the diameter of the network and $\gamma$ is the normalized eigengap of the interaction matrix. Even though extra information about the topology of the network is required, the work in \cite{sca17} provides a coherent understanding of the optimal convergence rates and its dependencies on the communication network. The work in~\cite{sca17} is based on the representation of the communication structure as an additional set of linear constraints on the distributed problem to guarantee consensus on the solution, from which a primal-dual method can be applied~\cite{parikh2014block,Zhu2010,khuzani2016distributed,zhang2018distributed,maros2018panda,Simonetto2016}. For example in~\cite{Latafat2016}, the authors develop a new primal-dual algorithm that uses the Laplacian of the communication graph as a set of linear constraints to induce coordination. Moreover, with additional metric subregularity conditions a linear convergence rate is shown. Recently in \cite{scaman2018optimal}, the authors extended the optimality lower bounds from \cite{sca17} to the non-smooth case using the randomized regularization approach~\cite{duchi2012}.
\end{itemize}
}

In this paper, we follow the approach in \cite{sca17} by formulating a dual problem and exploit recent results in the study of convex optimization problems with affine constraints \cite{ani17,dvu16b,gas17} to develop algorithms with provably optimal convergence rates for the cases where each of the objective functions $f_i$ has one the following properties: 
1) it is strongly convex and with Lipschitz continuous gradients; 
2) it is strongly convex and Lipschitz continuous (but not necessarily smooth); 
3) it is convex with Lipschitz continuous gradients, and 4) it is convex and Lipschitz continuous (not necessarily smooth), \cu{see Table~\ref{tab:summary_other}}. \cub{We say that a function $f$ is $M$-Lipschitz if $\|\nabla f(x)\|\leq M$, a function $f$ is $L$-smooth if $\|\nabla f(x) - \nabla f(y) \|_2 \leq L\|x-y\|_2$, a function $f$ is $\mu$-strongly convex ($\mu$-s.c.) if, for all $x, y \in \mathbb{R}^n$, $f(y) \geq f(x) + \langle \nabla f(x) ,y-x\rangle + \frac{\mu}{2}\|x-y\|^2$.}

\begin{table}
	\tbl{Iteration complexity of distributed optimization algorithms. All estimates are presented up to logarithmic factors, i.e. of the order $\tilde O$.}{
		\begin{tabular}{cccccc} \toprule
			{\bf Approach} &{\bf Reference} & \multicolumn{1}{C{1.7cm}}{\bf $\mu$-strongly convex and $L$-smooth}    & \multicolumn{1}{C{1.7cm}}{\bf  $\mu$-strongly convex and \mbox{$M$-Lipschitz}}   & {\bf $L$-smooth}   & {\bf $M$-Lipschitz}   \\ \toprule
Centralized	&\cite{Nemirovskii1983}	& $ \sqrt{\frac{L}{\mu}} $ &$ \frac{M^2}{\mu \varepsilon}  $ &   $ \sqrt{\frac{L}{\varepsilon} }$  &  $ \frac{M^2}{\varepsilon^2} $ \\ \colrule
\multirow{7}{*}{\parbox[c]{2.4cm}{\vspace{-0.9cm} \centering Gradient Computations}} 
&\cite{Qu2017}$^{\text{b}}$ &$m^3 \big(\frac{L}{\mu}\big)^{5/7} $ & $-$ & $ \frac{1}{\varepsilon^{5 / 7}}$ & $-$\\
&\cite{ols14} & $-$ & $-$& $-$& $ m\frac{M^2}{\varepsilon^2} $\\
&\cite{Duchi2012ab} & $-$ & $-$& $-$& { $ m^2\frac{M^2}{\varepsilon^2} $}\\
&\cite{Doan2017} & $m^2 \frac{L}{\mu}$ &$-$ & $-$ & $-$ \\
&\cite{Lakshmanan2008} 		&  $-$&$-$ & $m^3 \frac{L}{\varepsilon} $ & $-$  \\ 
&\cite{Necoara2013}	& $ m^4\frac{L}{\mu}$ &$-$ &   $ m^4\frac{L }{\varepsilon} $  &$-$\\
& \cite{jak17}$^{\text{c}}$ & $ m^2\sqrt{\frac{L}{\mu}}$ &$-$ &$-$ & $-$\\ \colrule
\multirow{2}{*}{\parbox[c]{2.2cm}{\vspace{0.3cm} \centering Communication Rounds}} &\cite{sca17}	&  $ m\sqrt{\frac{L}{\mu}} $ & $-$ & $-$ & $-$\\
&\cite{lan17}& $-$ &$m^2\sqrt{\frac{M^2}{\mu \varepsilon}} $ &   $-$  &$m^2\frac{M}{\varepsilon}$\\
&\cu{\textbf{This paper}}& \cu{$ \boldsymbol{m\sqrt{\frac{L}{\mu}} }$ } &\cu{$\boldsymbol{m\sqrt{\frac{M^2}{\mu \varepsilon}} }$ }&\cu{   $ \boldsymbol{m\sqrt{\frac{L }{\varepsilon} }} $ } &  \cu{ $\boldsymbol{m\frac{M}{\varepsilon}} $}\\ \botrule
		\end{tabular}
	}
\tabnote{\textsuperscript{a}Additionally, it is assumed functions are proximal friendly. No explicit dependence on $L$, $M$ or $m$ is provided.} 
\tabnote{	\textsuperscript{b}An iteration complexity of $\tilde{O}(\sqrt{1 / \varepsilon} ) $ is shown if the objective is the composition of a linear map and a strongly convex and smooth function. Moreover, no explicit dependence on $L$ and $m$ is provided.}
\tabnote{	\textsuperscript{c}A linear dependence on $m$ is achieved if $L$ is sufficiently close to $\mu$.}
	\label{tab:summary_other}
\end{table}

 \cu{Again, our goal is to investigate whether distributed algorithms,
	that use  local agents objective functions and local agents' interactions in a given communication graph, can match the performance of their centralized counterparts. With respect to this goal, the contributions of our work is in the development of the algorithms and their analysis. The major novelty is in the establishment of the results showing that the performance of our distributed algorithms can match the performance of their centralized counterparts, up to a logarithmic factor. These rate results are established for two classes of functions, namely, the functions 
	that are dual friendly and those that are not dual friendly. However, to accomplish these optimal rates (up to a logarithmic scaling) of centralized methods, some common knowledge of certain quantities about the underlying graph is needed. These quantities may require additional information preprocessing in the graph, which we do not investigate here; it is beyond the scope of the paper. Additionally, our convergence rate analysis is useful for the design of the communication graphs, as it indicates how the performance of our methods will depend on various quantities of the graph and the agents' objective functions.}

We provide a sequence of algorithms, to minimize functions from the problem classes described above, that achieve \cub{an $\varepsilon$-approximate solution (c.f. Definition~\ref{def:opt_sol})} for \textit{any} fixed, connected and undirected graph according to Table~\ref{tab:summary_other}, where universal constants, logarithmic terms, and dependencies on the initial conditions are hidden for simplicity. The resulting iteration complexities are given both for the optimality of the solution and the violation of the consensus constraints. Note that for distributed algorithms based on primal iterations these estimates translate to computations of gradients of the local functions for each of the agents. On the other side, in dual based algorithms, the complexity refers to computations of the gradients of the Lagrangian dual function, which translates to the number of communication rounds in the network. Initially, we consider dual friendly or admissible functions. Our results match known optimal complexity bounds for centralized convex optimization (obtained by classical methods such as Nesterov's fast gradient method \textsc{FGM} \cite{nes83}), with an additional cost induced by the network of communication constraints. This extra cost appears in the form of a multiplicative term proportional to the square root of the spectral gap of the interaction matrix. Later we study non-dual friendly functions where the number of gradient computations has to be taken into account. Therefore, we provide complexity bounds in terms of oracle calls, i.e., gradient computations, as well as communication rounds. Moreover, we provide a new algorithm for smooth functions that improves the dependency on the strong convexity parameters of the local functions.

This paper is organized as follows: we start with a couple of preliminary definitions and auxiliary results in Section~\ref{sec:prelim}. Section \ref{sec:problem} introduces the problem of distributed optimization over a network. \cub{Section~\ref{sec:pd_linear} presents a general form of primal-dual analysis of convex problems with linear constraints, this will serve as the basis for our main results. Section \ref{sec:main} provides our main results on the optimal convergence rates for distributed convex minimization problems with a dual friendly structure.} Section \ref{sec:non_dual} provides our main results on the optimal convergence rates for distributed convex minimization problems where an exact solution to the dual subproblem is not available. Section \ref{sec:improve_bound} presents a method to improve the dependency of the convergence rate on the condition number of the function $F$. Discussions are provided in Section~\ref{sec:discussion}.  Section \ref{sec:experiments} provides numerical experiments of the proposed algorithms. Finally, Section~\ref{sec:conclusions} provides some conclusions and points out future work.

	\noindent\textbf{Notation:} 
We will assume that the nodes in the network, also referred \cu{to} as agents, are indexed from $1$ through $m$ (no actual enumeration is needed in the execution of the proposed algorithms; it is only used in our analysis).  
We use the superscripts $i$ or $j$ to denote agent indices and the subscript $k$ to denote the iteration index of an algorithm. We denote by $[A]_{ij}$ the entry of the matrix $A$ in its $i$-th row and $j$-th column, and 
write $I_{n}$ for the identity matrix of size $n$. For a \cub{positive semi-definite matrix} 
$W$, we let $\lambda_{\max}(W)$ be its largest eigenvalue 
and $\lambda_{\min}^+(W)$ be its smallest positive eigenvalue, and we denote its condition number 
by $\chi(W) = \lambda_{\max}{(W)}/ \lambda_{\min}^{+}{(W)}$.
Given a matrix $A$, define $\sigma_{\max}(A)\triangleq \lambda_{\max}(A^TA)$ and
$\sigma_{\min}^+(A) \triangleq \lambda_{\min}^+ (A^TA)$. 
We use $\boldsymbol{1}$ to denote a vector with all entries equal to 1.
We write $\tilde{O}$ to denote a complexity bound that ignores logarithmic factors. We will work in the standard Euclidean norm, denoted by $\|\cdot\|_2$.  
	
	\section{Preliminaries}\label{sec:prelim}
	
	In this section, we provide some definitions and preliminary information that we will use in the forthcoming sections. 

%

If each function $f_i$ in~\eqref{consensus_problem2} 
is $\mu_i$-strongly convex in $x_i$, then $F$ in~\eqref{consensus_problem2} is $\mu$-strongly convex in $x$ with \mbox{$\mu = \min_{1\le i\le m} \mu_i$}. Also, if each $f_i$ is \mbox{$L_i$-smooth}, then $F$ is $L$-smooth with \mbox{$L = \max_{1\le i\le m} L_i$}. In Section~\ref{sec:improve_bound} \cub{we explore} a method to improve the dependency on the condition number $\mu$ from the worst case parameter to the average strong convexity.

We will build the proposed algorithms based on Nesterov's fast gradient method (FGM)~\cite{nes13}. For example, a version of the \textsc{FGM} for a $\mu$-strongly convex and $L$-smooth function $f$ is shown in \cub{Algorithm~\ref{alg:nesterov}}. Note that Theorem~\ref{thm:nesterov} is precisely Theorem $2.2.\cu{3}$ in~\cite{nes13} Other variants of this method can be found in~\cite{nes13,bec09,lan11}.
\begin{algorithm}[t]
	\caption{Nesterov's Constant Step Scheme II.}
	\begin{algorithmic}[1]
		\STATE{Choose $x_0\in \mathbb{R}^n$ and $\alpha_0 \in (0,1)$. Set $y_0 = x_0$ and $q = \frac{\mu}{L}$.}
		\STATE{ $k$th iteration $(k\geq 0)$.
			\begin{enumerate}
				\item[(a)] Compute $\nabla f(y_k)$. Set $ x_{k+1} = y_k - \frac{1}{L}\nabla f(y_k)$.
				\item[(b)] Compute $\alpha_{k+1} \in (0,1)$ from equation  
				$\alpha_{k+1}^2 = (1-\alpha_{k+1})\alpha_k^2 + q \alpha_{k+1}$, \newline
				and set $\beta_k = \frac{\alpha_{k}(1-\alpha_{k})}{\alpha_{k}^2 + \alpha_{k+1}}$, 
				$ y_{k+1} = x_k  + \beta_k (x_{k+1} - x_k)$.
			\end{enumerate}    
		}
	\end{algorithmic}
	\label{alg:nesterov}
\end{algorithm}

\cu{
\begin{theorem}[Adapted from Theorem $2.2.3$ in~\cite{nes13}]\label{thm:nesterov}
	If in Algorithm~\ref{alg:nesterov} one sets $\alpha_{0}$ such that $L(1-\alpha_0) = \alpha_0(\alpha_0L - \mu)$, then Algorithm~\ref{alg:nesterov} generates a sequence $\{x_k\}_{k=0}^{\infty}$ such that
	\begin{align}\label{rate_nest}
	f(x_k) - f^* \leq L\left( 1- \sqrt{\frac{\mu}{L}} \right)^k\|x_0 - x^*\|_2^2, 
	\end{align}
	where $f^*$ denotes the minimum value of the function $f$ over $\mathbb{R}^n$ and $x^*$ is its 
	minimizer. Moreover, Algorithm~\ref{alg:nesterov} is optimal for unconstrained minimization of strongly convex and smooth functions.
\end{theorem}
}


	\section{Problem Statement}\label{sec:problem}
	
	Defining the stacked column vector $x = [x_1^T,x_2^T,\hdots,x_m^T]^T \in \mathbb{R}^{mn}$ allows us to rewrite problem~\eqref{main_problem} in an equivalent form as follows:
\begin{align}\label{equiv_main_problem}
\min_{x_1 = \hdots =x_m} {F(x)} & \qquad 
\text{where } \qquad F(x)\triangleq\sum_{i=1}^{m}f_i(x_i).
\end{align}

The distributed optimization framework assumes we want to solve problem~\eqref{equiv_main_problem} in a distributed manner over a network. 
We model such a network as a fixed connected undirected graph $\mathcal{G} = (V,E)$, where 
$V$ is the set of $m$ nodes, and $E$ is the set of edges. We assume that the graph $\mathcal{G}$ does
not have self-loops.
The network structure imposes information constraints; specifically, each node $i$ 
has access to the function $f_i$ only and a node can exchange information only with its immediate neighbors, i.e., a node $i$ can communicate with node $j$ if and only if $(i,j)\in E$.

The communication constraints imposed by the network can be represented as
a set of equivalent constraints via the Laplacian  
$\bar W{\in \mathbb{R}^{m\times m}}$ of the graph $\mathcal{G}$, defined as: $[\bar W]_{ij} = -1$ if $(i,j) \in E$, $[\bar W]_{ij} = \text{deg}(i)$ if $i= j$, and $[\bar W]_{ij} = 0$ otherwise,
where $\text{deg}(i)$ is the degree of the node $i$, i.e., the number of neighbors of the node. Moreover, define the communication matrix (also referred to as an interaction matrix) 
by \mbox{$W \triangleq\bar W \otimes I_n$}, where $\otimes$ indicates the Kronecker product.

Throughout this paper, {\it we assume the graph $\mathcal{G} = (V,E)$ is connected and undirected}.
Under this assumption, the Laplacian matrix $\bar W$ is symmetric and positive semi-definite. Furthermore,
the vector $\boldsymbol{1}$ is the unique (up to a scaling factor) eigenvector associated with the eigenvalue
$\lambda=0$. One can verify that $W$ inherits all the properties of $\bar W$, i.e., it is a symmetric positive semi-definite matrix and it satisfies the following relations: $W{x} = 0$ if and only if {$x_1 = \hdots = x_m$}, $\sqrt{W}{x} = 0$ if and only if {$x_1 = \hdots = x_m$}, and $\sigma_{\max} (\sqrt{W}) = \lambda_{\max} (W)$. Therefore, one can equivalently rewrite problem~\eqref{equiv_main_problem} as follows:
\begin{align}\label{consensus_problem2}
\min_{\sqrt{W} x=0} F(x) & \qquad 
\text{where } \qquad F(x) \triangleq\sum_{i=1}^{m}f_i(x_i).
\end{align}

Note that the constraint sets $\{x\mid \sqrt{W} x=0\}$ and $\{x\mid x_1 = \hdots  = x_m\}$ are equal. This is because $\ker (\sqrt{W}) = \spn(\boldsymbol{1})$ due to  $\mathcal{G}$ being connected. A similar idea of writing the Laplacian as a product of a matrix $B$ and its transpose has been employed in~\cite{Latafat2016} using the incidence matrix.

Our results provide convergence rate estimates for the solution of the problem in~\eqref{main_problem} for four different cases in terms of the properties of the function \newline \mbox{$F(x)=\sum_{i=1}^m f_i(x_i)$}.
\begin{assumption}\label{assumptions}
	Consider a function $F(x)=\sum_{i=1}^m f_i(x_i)$, and assume:
	\begin{enumerate}
		\item[(a)] Each $f_i$ is $\mu_i$-strongly convex and $L_i$-smooth, thus $F$ is $\mu$-strongly convex and $L$-smooth;
		\item[(b)] each $f_i$ is $\mu_i$-strongly convex and $M_i$-Lipschitz on a ball around the optimal point $x^*$ with radius \cub{$\|x^*-x^0\|_2$, where $x^0$ is an initialization point}, thus $F$ is \mbox{$\mu$-strongly} convex and $M$-Lipschitz on the same bounded set;
		\item[(c)] each $f_i$ is convex and $L_i$-smooth, thus $F$ is convex and $L$-smooth;
		\item[(d)] each $f_i$ is convex and $M_i$-Lipschitz on a ball around the optimal point $x^*$ with radius \cub{$\|x^*-x^0\|_2$, where $x^0$ is an initialization point}, thus $F$ is convex and \mbox{$M$-Lipschitz} on the same bounded set;
	\end{enumerate}
	where $\mu=\min_{1\le i\le m}\mu_i$, $L=\max_{1\le i\le m}L_i$ and \mbox{$M=\max_{1\le i\le m}M_i$}.
\end{assumption}

	\begin{example}


\begin{figure}
	\centering
	\subfigure[Cycle graph.]{%
		\includegraphics[width=0.21\textwidth]{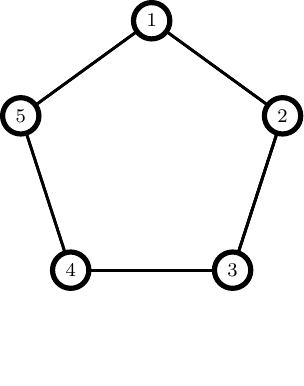}}\hspace{50pt}
	\subfigure[Erd\H{o}s-R\'enyi random graph.]{%
		\includegraphics[width=0.3\textwidth]{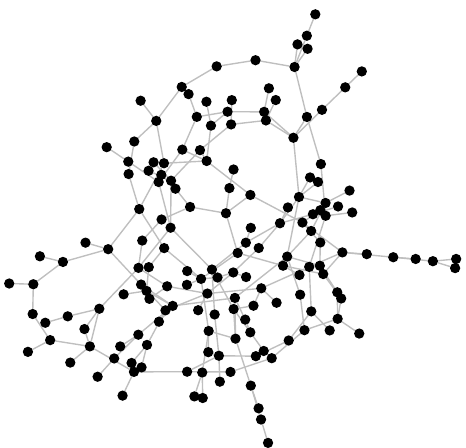}}
	\caption{Two examples of networks of agents. (a) A cycle graph with $5$ agents. (b) An Erd\H{o}s-R\'enyi random graph with $160$ agents.} \label{fig:network}
\end{figure}
	
	Consider a network of agents as shown in Figure~\ref{fig:network}, where the agents in the network seek to cooperatively solve the following regularized linear regression problem
	\begin{align}\label{ex:cost1}
	\min_{z \in \mathbb{R}^n} \frac{1}{2ml}\|Hz - b\|_2^2 + \frac{1}{2}c \|z\|_2^2,
	\end{align}
	where $b \in \mathbb{R}^{ml}$, $H \in \mathbb{R}^{ml \times n}$ and $c>0$ is some constant. Furthermore, assume the data in $b$ and $H$ are distributed over the network, where no single agent has full access to the complete information, i.e., each agent has access to a subset of points such that
	\begin{align*}
	b^T & = [\underbrace{b_1^T }_{\text{Agent} \ 1}\mid \underbrace{b_2^T}_{\text{Agent} \ 2}\mid \cdots \mid \underbrace{b_m^T}_{\text{Agent}  \ m}],
\ \ \ \text{and} \ \ \
	H^T = [\underbrace{H_1^T }_{\text{Agent} \ 1}\mid \underbrace{H_2^T}_{\text{Agent} \ 2}\mid \cdots \mid \underbrace{H_m^T}_{\text{Agent}  \ m}],
	\end{align*}
	where $b_i \in \mathbb{R}^l$ and $H_i \in \mathbb{R}^{l\times n}$ for each $i$.	Therefore, problem~\eqref{ex:cost1} is equivalent to 
	\begin{align}\label{ex:ridge}
	\min_{\sqrt{W}x =0} \sum_{i=1}^{m} \left( \frac{1}{2}\frac{1}{ml}\|b_i - H_ix_i\|_2^2 + \frac{1}{2}\frac{c}{m} \|x_i\|_2^2\right),
	\end{align}
	where $W = \bar W \otimes I_n$. Particularly, for the cycle graph network of $5$ agents shown in Figure \ref{fig:network}(a), agent $1$ can share information with agents $2$ and $5$, agent $5$ shares information with agents $1$ and $4$, and similarly for the other agents. 
	
	
\end{example}
	
	\section{Primal-Dual Analysis for Convex Problems with Linear Constraints}\label{sec:pd_linear}
	
	In this section, we consider a generic optimization problem with linear constraints. We will use the convergence results in~\eqref{rate_nest} to obtain some fundamental insights for the distributed optimization problem. Moreover, we will derive the results for a corresponding distributed algorithm for solving problem~\eqref{consensus_problem2}. The main idea of our analysis will be to explore the case when the linear constraints $Ax=0$ represent the network communication constraints as $\sqrt{W}x=0$ and the function $f(x)$ corresponds to the network function $F(x)$ as defined in~\eqref{consensus_problem2}. 

Initially, consider a $\mu$-strongly convex and an $L$-smooth function $f$ 
to be minimized over a set of linear constraints, i.e.,
\begin{align}\label{main2}
\min_{Ax=0} f(x).
\end{align} 

Assume that problem~\eqref{main2} is feasible, in which case a unique solution exists, denoted by $x^*$. However, we will be interested in finding approximate solutions of~\eqref{main2} that attain a function value arbitrarily close to the optimal value and have arbitrarily small feasibility violation
of the linear constraints. For this, we introduce the following definition.

\begin{definition}\cite{lan17}\label{def:opt_sol}
	A point $x \in \mathbb{R}^{mn}$ is called an $(\varepsilon,\tilde{\varepsilon})$-solution of~\eqref{main2} if the following conditions are satisfied: $
	f(x) - f^*  
	\le  \varepsilon$, and $\|Ax\|_2 \leq \tilde{\varepsilon}$,
	where $f^* = f(x^*)$ denotes the optimal value for the primal problem in~\eqref{main2}. 
\end{definition}


The Lagrangian dual of~\eqref{main2} is given by
\begin{align}\label{eq:laggrange_dual}
\min_{Ax=0} f(x) & = \max_y \left\lbrace \min_{{x}} \left\{ f(x) - \left\langle A^Ty,x \right\rangle \right\} \right\rbrace.
\end{align}
Moreover, \eqref{eq:laggrange_dual} can be re-formulated as an equivalent minimization problem, as follows:
\begin{align}\label{dual_problem22}
\min_y  \varphi(y) \quad \text{where} \quad  \varphi(y) \triangleq  \max_x \Psi(x,y),
\end{align}
and $
\Psi(x,y) \triangleq  \left\langle   A^Ty,x \right\rangle -f(x)$.

The function 
$\varphi(y)$ is $\mu_\varphi$-strongly convex on $\ker(A^T)^{\perp}$ with
$\mu_\varphi = \lambda_{\min}^{+}(A^TA)/L$. Moreover,  it has 
$L_\varphi$-Lipschitz continuous gradients with $L_\varphi= \lambda_{\max}(A^TA)/\mu$, see Lemma~$3.1$ in \cite{Beck2014}, Proposition $12.60$ in~\cite{rockafellar2011variational}, Theorem $1$ in~\cite{nes05}, Theorem~$6$ in~\cite{Kakade2009a}.
Additionally, from Demyanov--Danskin's theorem 
(see, for example, Proposition $4.5.1$ in \cite{Bertsekas2003}),  
it follows that 
\ag{
\begin{align}\label{eq:grad}
\nabla \varphi(y) = Ax^*(A^Ty),
\end{align}}
where $x^*(A^Ty)$ denotes the unique solution 
\cub{of} the inner maximization problem 
\begin{align}\label{eq:xstary}
x^*(A^Ty) & = \argmax_x \Psi(x,y).
\end{align}
Also known as the \textit{sharp}-operator~\cite{yurtsever2015universal}.

We will call $x^*$ the minimizer of~\eqref{main2}. On the other hand, we denote $x^*(A^Ty)$ as the solution of~\eqref{eq:xstary} for a given value $A^Ty$. Particularly, $x^*(0) = \argmax_x \left\lbrace -f(x)\right\rbrace $. Moreover, there is no duality gap between the primal problem in~\eqref{main2} and its dual problem in \eqref{dual_problem22}, and 
the dual problem has a solution $y^*$ (see for example, Proposition $6.4.2$ in~\cite{Bertsekas2003}).
Thus, it holds that $x^* = x^*(A^Ty^*)$.
In general, the dual problem in~\eqref{dual_problem22} can have multiple solutions of the form 
$y^* + \ker(A^T)$ when the matrix $A$ does not have a full row rank, for example when $A$ is the Laplacian of a graph. If the solution is not unique, 
then  $y^*$ denotes the {\it smallest norm} solution, and we let $R$ be its norm, i.e.
$R = \|y^*\|_2$. 



\cu{ Note that we typically do not know $R$ or $R_x$. Thus, we require a method to estimate the strong convexity parameter, which is challenging \cite{nes07,odo15}. Some recent work have explored restarting techniques to reach optimal convergence rates when the strong convexity parameters are unknown \cite{odo15,iou14}. Similarly, a generalization of the \textsc{FGM} algorithm can be proposed when the smoothness parameter is unknown \cite{gas16b}. However, the effect of restarting in the distributed setup requires further study and is out of the scope of this paper. Additionally, parameters such as  $\lambda_{\min}^{+}(W)$, $\lambda_{\max}(W)$ can be efficiently computed in a distributed manner~\cite{TRAN20145526}. Moreover, the values of~$\mu$ and $L$ can be shared with simple max-consensus which is guaranteed to converge in finite time~\cite{maxcon}.}

In the next Lemma, we provide an auxiliary condition to check whether a point $x$ is an  $(\varepsilon,\tilde{\varepsilon})$-solution in terms of the properties of the dual function $\varphi$.
\begin{lemma}[Lemma $1$ in \cite{gas16b}]
	Let $\langle y, \nabla \varphi (y) \rangle \leq \varepsilon $ and $\|\nabla \varphi (y)\|_2 \leq \tilde{\varepsilon}$. Then, $x^*(A^Ty)$ is an $(\varepsilon,\tilde{\varepsilon})$-solution of~\eqref{main2}.
	\label{lemma:sasha}
\end{lemma} 

\cu{
\begin{remark}
	From a technical perspective, dual smoothing with linear constraints is not new. However, the main technical contribution of our approach is the elimination of the assumption that the target function has bounded variation on the admissible set (that is typically unbounded). \cub{That is, for a function $f(\cdot)$ defined on a set $Q$, it holds that $\max_{x,y \in Q} \{f(x) -f(y)\} \leq \Delta < \infty$~\cite{gas16b,ani17}, or a \ag{weaker condition where if we define $\Delta(f)=\max_{x,\ag{y}\in Q}\{\nabla f(x)[y-x]\}$, where $\nabla f(x)[y-x]$ is the directional derivative of $f$ at a point $x$ in the direction $y-x$ \cite{dev12}}, then $\|x^*\|\leq \Delta(f)/r$, with $r$ is the radius of a ball inside $Q$, see \cite[Theorem 6.1]{dev12}.} One can verify that this assumption significantly required in the previous dual smoothing (regularized) works. We show how to obtain all the results without such a restrictive assumption. As far as we known this is the first time such strong condition is removed.
\end{remark}
}

In what follows, we will apply the bound for the \textsc{FGM} algorithm in~\eqref{rate_nest} on the dual problem~\eqref{dual_problem22}, which is not strongly convex in the ordinary sense (on the whole space), see Algorithm~\ref{alg:nesterov_dual}. 
However, by choosing
$y_0 = x_0=0$ in Algorithm~\ref{alg:nesterov} as the initial condition, 
the algorithm applied to the dual problem will produce iterates that lie 
in the linear space of gradients 
$\nabla \varphi(y)$, which are of the form $Ax$ for $x=x^*(A^Ty)$. In this case, the dual function $\varphi(y)$ will be strongly convex when $y$ is restricted to the linear space spanned by the range of the matrix $A$.

\cub{\begin{algorithm}[t]
	\caption{Nesterov's Constant Step Scheme II for the dual problem.}
	\begin{algorithmic}[1]
		\STATE{\cub{Choose $y_0\in \mathbb{R}^n$ and $\alpha_0 \in (0,1)$. Set $\tilde{y}_0 = y_0$ and $q = {\mu_\varphi}/{L_\varphi}$.}}
		\STATE{ \cub{$k$th iteration $(k\geq 0)$.}
			\begin{enumerate}
				\item[(a)] \cub{Compute $\nabla \varphi(y_k)$. Set $ y_{k+1} = \tilde y_k - \frac{1}{L_\varphi}\nabla \varphi(\tilde y_k)$.}
				\item[(b)] \cub{Compute $\alpha_{k+1} \in (0,1)$ from equation  
				$\alpha_{k+1}^2 = (1-\alpha_{k+1})\alpha_k^2 + q \alpha_{k+1}$, \newline
				and set $\beta_k = \frac{\alpha_{k}(1-\alpha_{k})}{\alpha_{k}^2 + \alpha_{k+1}}$, 
				$ \tilde y_{k+1} = y_k  + \beta_k (y_{k+1} - y_k)$.}
			\end{enumerate}    
		}
	\end{algorithmic}
	\label{alg:nesterov_dual}
\end{algorithm}
}

Note we are now required to compute explicitly $x^*(A^Ty)$, which is the solution of the inner maximization problem~\eqref{eq:xstary}. Section~\ref{sec:main} will present the proposed algorithms and convergence rates, for the different convexity and smoothness assumptions on the functions $f_i$ expressed in Assumption~\ref{assumptions}, when an explicit solution to the inner maximization problem~\eqref{eq:xstary} is available. We will denote this scenario as dual-friendly functions. Particularly, to find $x^*(A^Ty)$ one can use optimal (randomized) numerical methods~\cite{nes13,nes94,bub15}. Later in Section~\ref{sec:non_dual}, we extend the results of Section~\ref{sec:main} to the case where only approximate solutions to~\eqref{eq:xstary} can be computed. 

\begin{definition}\label{def:dual_friendly}
	A function $f(x)$ is \textit{dual-friendly} if, for any $y$, 
	a solution  $x^*(A^Ty)$ in~\eqref{eq:xstary} can be computed ``efficiently" (e.g., by a closed form or by polynomial time algorithms). Sometimes this function are also called \textit{admissible}~\cite{hiriart2012fundamentals,Raginsky2012}, or Fenchel tractable~\cite{Dunner:2016}.
\end{definition}

Several authors have previously studied the primal-dual approach we will follow in this paper. In~\cite{yurtsever2015universal}, the authors propose a universal framework for the study of constrained optimization problems. Leveraging on the Fenchel operator, the authors provide an algorithm that can optimally adapt to an unknown H\"older continuity parameters. In~\cite{Dunner:2016}, the authors provide an algorithm-free generic framework that provides convergence rate certificates for primal-dual algorithms. In~\cite{tran2015splitting,tran2014constrained}, the authors study a model-based gap and split-gap reduction techniques that simultaneously updates the primal and dual smoothness parameters for the convergence to the duality gap at optimal rates. In contrast, we focus our analysis on the recent results in accelerated primal-dual approaches that allows us to obtain optimal rates for larger classes of functions and quantify the constraint violation and distance to optimality~\cite{gas16b,dvu16b}.

Examples of optimization problems for which Definition \ref{def:dual_friendly} holds can be found in the literature, e.g., the entropy-regularized optimal transport problem \cite{cut16}, the entropy linear programming problem \cite{gas16b} or the ridge regression. Note that by definition, finding a solution for the problem~\eqref{eq:xstary} corresponds to finding a maximizing point of the Legendre transformation $f^*$ of the function $f$, where $f^*$ is defined as
\begin{align*}
f^*(y) & = \max_x \left\lbrace  \langle x, y \rangle -f(x) \right\rbrace.
\end{align*}
Moreover, if the conjugate dual function is available, then the maximizing argument is $
x^*(A^Ty) = \nabla f^*(y)$, see Proposition $8.1.1$ in \cite{Bertsekas2003}. For example, for the ridge regression problem~\eqref{ex:ridge},
the maximizing argument in~\eqref{eq:xstary} can be explicitly computed as $x^*(A^Ty)  = (H^TH)^{-1}(A^Ty +H^Tb)$.

Another example where one can find an explicit solution to the auxiliary dual problem is the Entropy Linear Program (ELP) \cite{ani15}, i.e.,
\begin{align}\label{prob:ELP}
\min_{x \in S_n(1), Ax = 0} \sum\limits_{j=1}^{n}x_j \log\left(\frac{x_j}{q_j} \right),  
\end{align}
where $S_n(1) = \{ x \in \mathbb{R}^n : x_j \geq 0 ; j=1,2,\hdots,n; \sum_{j=1}^{n}x_j = 1 \}$ is a unit simplex in $\mathbb{R}^n$ and $q \in S_n(1)$. The maximizing argument~\eqref{eq:xstary} for problem~\eqref{prob:ELP} can be explicitly computed as
\begin{align}
[x^*(A^Ty)]_i & = \frac{q_i \exp ([A^Ty]_i)}{\sum_{j=1}^{n} q_j \exp ([A^Ty]_j)}.
\end{align}

Additional examples related to optimal transport problems of computation of Wasserstein barycenter can be found in \cite{cut16,Uribe2018,Dvurechensky2018a}.  

\cub{
\begin{remark}
	Note that the particular example of Entropy Linear Program in~\eqref{prob:ELP}, the function is not defined on $\mathbb{R}^n$ as in~\eqref{main_problem}. In general, we can define the function $f_i: \mathcal{E}\to \mathbb{R}$, and the optimization space to be constrained to $x\in Q$, where $\mathcal{E}$ is some finite-dimensional vector space and $Q$ is a simple closed convex set. In that case, the matrix $A$ is defined as a linear operator from $E$ to another finite-dimensional real vector space $H$, with $0 \in H$. Moreover, strong convexity should be defined on $Q$ with respect to some chosen norm $\|\cdot\|_\mathcal{E}$, which implies that any subgradient of $f_i$ is an element of the dual space $\mathcal{E}^*$, with dual norm $\|\cdot\|_{\mathcal{E}^*}$.
	\begin{align*}
	\|g\|_{\mathcal{E}^*} & = \max_{\|x\|_\mathcal{E} \leq 1} \langle g,x \rangle.
	\end{align*}
	Similarly, for the linear operator $A: \mathcal{E}_1 \to \mathcal{E}_2$, $\lambda_{\max}$ and should be appropriately defined as
	\begin{align*}
	\lambda_{\max}(W) & = \|A\|_{\mathcal{E}_1 \to \mathcal{E}_2} = \max_{x\in \mathcal{E}_1, u \in \mathcal{E}_2} \left\lbrace \langle u, Ax \rangle : \|x\|_{\mathcal{E}_1} = 1, \|u\|_{\mathcal{E}_2^*} = 1 \right\rbrace,
	\end{align*} 
	and similarly for $\lambda_{\min}^+$. However, for simplicity of exposition, we will provide our results in the Euclidean space.
\end{remark}
}

\begin{remark}
	\cub{Without loss of generality, we assume that the optimal points $x^*\neq 0$ and $y^* \neq 0$. This facilitates notation, and simplifies analysis as we initialize our algorithms at the $0$ point, and our bounds will depend on the distance between the initial point and the optimal point. }
\end{remark}

	\section{Dual Friendly Functions: Algorithms and Iteration Complexity Analysis}\label{sec:main}
	
	In this section, we provide our main results about the communication complexity considering each function class in Assumption~\ref{assumptions}. For each case, we present an algorithm and the minimum number of iterations required fo the algorithm to reach an approximate solution of~\eqref{consensus_problem2}. 

\subsection{Sums of Strongly Convex and Smooth Functions}

In this subsection, we analyze the distributed optimization problem in \eqref{equiv_main_problem} when Assumption \ref{assumptions}(a) holds, i.e., each function $f_i$ is strongly convex and smooth. 

Initially, consider step (a) in Algorithm~\ref{alg:nesterov_dual}, and replace $A = \sqrt{W}$, thus
	\begin{align}\label{eq:nes_sqrt}
	y_{k{+}1} &= \ag{\tilde{y}_k {-} \frac{1}{L_\varphi}\nabla\varphi(\tilde{y}_k) = } \tilde{y}_k {-} \frac{1}{L_\varphi} \sqrt{W}x^*(\sqrt{W}^T\tilde{y}_k).
	\end{align}  
Unfortunately,~\eqref{eq:nes_sqrt} cannot be executed over a network in a distributed manner because the sparsity pattern of the matrix $\sqrt{W}$ need not be compliant with the graph $\mathcal{G}$ in the same way the matrix $W$ is. Therefore, we make the following change of variables: $\sqrt{W}y_k = z_k$ and \mbox{$\sqrt{W}\tilde y_k = \tilde z_k$}, resulting in an algorithm that can be executed in a distributed manner. Interaction between agents is dictated by the term $Wx^*(\tilde z_k)$ which depends only on local information. As a result, each agent $i$ in the network has its local variables $z_k^i$ and $\tilde z_k^i$, and to compute their value at the next iteration, it only requires the information sent by the neighbors defined by the communication graph $\mathcal{G}$. Additionally, the dual subproblem can be computed in a distributed manner at node $i$ as
\begin{align}\label{eq:inner_problem}
x^*_i(\tilde z^i_k) & = \argmax_{x_i} \left\lbrace \left\langle \tilde z^i_k,x_i \right\rangle {-}f_i(x_i) \right\rbrace.
\end{align}

Next, we formally state the \textsc{FGM} algorithm applied to the dual of problem \eqref{equiv_main_problem} with the change of variable that \cub{allows} a distributed execution.

\begin{algorithm}[t]
	\caption{Distributed \textsc{FGM} for strongly convex and smooth problems}
	\begin{algorithmic}[1]
		\STATE{All agents set $z_0^i = \tilde{z}_0^i =0 \in \mathbb{R}^n$, $q =\frac{\mu}{L}\frac{\lambda^{{+}}_{\min}(W)}{\lambda_{\max}(W)}$, $\alpha_0$ solves $\frac{\alpha_0^2 {-}q}{1{-}\alpha_0} = 1$ and $N$.}
		\STATE{For each agent $i$}
		\FOR{ $k=0,1,2,\cdots,N{-}1$ }
		\STATE{$x^*_i(\tilde z^i_k)  = \argmax\limits_{x_i} \left\lbrace \left\langle \tilde z^i_k,x_i \right\rangle {-}f_i(x_i) \right\rbrace$}
		\STATE{Share $x^{*}_i(\tilde z_k^i)$ with neighbors, i.e. $\{j \mid (i,j) \in E \}$.}
		\STATE{{ $z_{k{+}1}^i = \tilde z_k^i {-} \frac{\mu}{\lambda_{\max}(W)} \sum_{j=1}^{m} W_{ij} x^{*}_j(\tilde z_k^j)$}}
		\STATE{Compute $\alpha_{k{+}1}\in (0,1)$ from $\alpha_{k{+}1}^2 = (1 {-} \alpha_{k{+}1})\alpha_{k}^2 {+}q \alpha_{k{+}1}$ and set $\beta_k = \frac{\alpha_k(1{-}\alpha_k)}{\alpha_{k}^2 {+} \alpha_{k{+}1}}$}
		\STATE{{ $\tilde z_{k{+}1}^i  = z_{k{+}1}^i {+} \beta_k(z_{k{+}1}^i {-} z_k^i)$}}
		\ENDFOR
	\end{algorithmic}
	\label{alg:case1}
\end{algorithm}

\cu{
\begin{remark}
	Note that the networked communication between agents is executed in Line~$5$ of Algorithm~\ref{alg:case1}. In order for an agent $i$ to execute line $6$, it needs to have access to $x^{*}_j(\tilde z_k^j)$ for the non{-}zero entries of $W_{ij} $.  Sharing with neighbors refers to the action of an agent $i$ sending a value it holds in its local memory, to the set of neighbors it can communicate with according to the network structure, i.e., $\{j \mid (i,j) \in E \}$. \cu{Also, note that the sequences $\{\alpha_k\}$ and $\{\beta_k\}$ can be computed independently by each agent and coordination is not needed.}  
\end{remark}
}

Algorithm~\ref{alg:case1} requires the number of iterations $N$, which effectively corresponds to the number \cub{of} communication rounds. We define a communication round as an iteration of Algorithm~\ref{alg:case1} where every node shares its local estimates with its neighbors and updates its local variables. Thus, we are interested in finding a lower bound on $N$ such that we can guarantee certain optimality of the local solutions in the sense of Definition \ref{def:opt_sol}. Next, we state our main result regarding the number of iterations required by Algorithm \ref{alg:case1} to reach an approximate solution of problem~\eqref{equiv_main_problem}.

\begin{theorem}\label{thm:case1} 
	Let $F(x)$ be dual friendly and Assumption~\ref{assumptions}(a) hold. 
	For any $\varepsilon >0$, the output $x^*(z_N)$ of Algorithm~\ref{alg:case1} is an $(\varepsilon,\varepsilon/R )${-}solution of~\eqref{consensus_problem2} for 
	\begin{align*}
	N & \geq 2\sqrt{\frac{L}{\mu}\chi(W)} \log \left(\frac{2\sqrt{2} \lambda_{\max}(W) R^2}{ \mu \cdot \varepsilon}  \right),
	\end{align*}
	where $R = \|y^*\|_2$, and $\chi(W) = \lambda_{\max}{(W)}/ \lambda_{\min}^{{+}}{(W)}$. 
\end{theorem}

\begin{proof}
	Algorithm~\ref{alg:case1} follows from Algorithm~\ref{alg:nesterov} applied to the dual problem~\eqref{dual_problem22} with the change of variables $\sqrt{W}y_k = z_k$ and \mbox{$\sqrt{W}\tilde y_k = \tilde z_k$}. \ag{In other words, we apply Algorithm~\ref{alg:nesterov_dual} to dual problem, i.e., Algorithm~\ref{alg:case1} is just Algorithm~\ref{alg:nesterov_dual} in new variables  $z_k = \sqrt{W}y_k$ and \mbox{$\tilde z_k = \sqrt{W}\tilde y_k$}.}  Therefore, we are going to use the convergence results of the \textsc{FGM} for the dual problem in terms of the dual variables $y_k$ and $\tilde y_k$ and provide an estimate of the convergence rate of in terms of the primal variables. 
	
	Initially, it follows from Theorem~\ref{thm:nesterov} that the sequence of estimates generated by the iterations in~\eqref{eq:nes_sqrt} has the following property:
		\begin{align}\label{nest_dual}
		\varphi(y_k){-} \varphi^* & \leq L_{\varphi}R^2 \exp\left({-}k \sqrt{\frac{\mu_{\varphi}}{L_{\varphi}}}\right).
		\end{align}
Moreover, it holds that 
	$$
	\|\nabla \varphi (\ag{y_k})\|_2^2  \leq  2 L_{\varphi} \left(\varphi(y_k) {{-}} \varphi^*\right).$$ Hence \ag{from \eqref{eq:grad}, \eqref{eq:xstary}, \eqref{nest_dual} with $A = \sqrt{W}$ ($W=W^T$)}  $$
	\|\sqrt{W} x^*(\sqrt{W} \ag{y_k})\|_2^2 \leq 2 L_{\varphi}^2R^2 \exp\left({{-}}k \sqrt{\frac{\mu_{\varphi}}{L_{\varphi}}}\right).$$
	
We can conclude that $\|\sqrt{W} x^*(z_k)\|_2\leq \varepsilon / R$ if	
$
	k \geq 2 \sqrt{\frac{L_\varphi}{\mu_\varphi}} \log \left(\frac{ \sqrt{2}L_\varphi R^2}{ \varepsilon} \right) 
$. 	Now, by using the Cauchy{-}{-}Schwarz inequality, it follows that	
	\begin{align*}
	|\langle y_k,\sqrt{W} x^*(\sqrt{W} y_k) \rangle|^2 & \leq \|y_k\|_2^2\|\sqrt{W} x^*(\sqrt{W} y_k)\|_2^2.
	\end{align*}
	
	\cub{The authors~\cite{dev12} showed that the iterates generated by Nesterov's fast gradient method remain in a ball around the optimal point, that can be bounded by the distance between the initial point and the optimal point. We exploit this property to bound $\|y_k\|_2$ as} \ag{\cite{dev12}}
	\begin{align*}
	\|y_k {-} y^*\|_2 &\leq \|y_0 {-} y^*\|_2.
	\end{align*}
	Thus, since we assume $y_0 = 0$, it holds that $\|y_k\|_2 \leq 2\|y^*\|_2 \leq 2R$, then
	\begin{align*}
	|\langle y_k,\sqrt{W} x^*(\sqrt{W} y_k) \rangle|^2 & \leq 4 R^2 \|\sqrt{W} x^*(\sqrt{W} y_k)\|_2^2 \ag{\leq} 8 R^4  L_{\varphi}^2 \exp\left({-}k \sqrt{\frac{\mu_{\varphi}}{L_{\varphi}}}\right).
	\end{align*}
	
	Therefore \ag{from Lemma 4.1} $f(x^*(z_k)) {-} f^* \ag{= |\langle z_k, x^*(z_k) \rangle| \leq |\langle y_k,\sqrt{W} x^*(\sqrt{W} y_k) \rangle| } \leq \varepsilon$ if
	\ga{$k \geq 2\sqrt{\frac{L_\varphi}{\mu_\varphi}} \log \left( \frac{2 \sqrt{2} L_\varphi^2 R^3}{ \varepsilon} \right)$}. Finally, based on Lemma~\ref{lemma:sasha}, Algorithm \ref{alg:case1} will produce an $(\varepsilon, \varepsilon/R)${-}solution if
\ga{	\begin{align*}
	N & \geq 2\sqrt{\frac{L_\varphi}{\mu_\varphi}} \log \left(\max\left\lbrace  \frac{2\sqrt{2} L_\varphi R^2}{ \varepsilon}, \frac{\sqrt{2} L_\varphi R^2}{\varepsilon} \right\rbrace  \right).
	\end{align*}}
	
	Following the definitions of $L_\varphi$, $\mu_{\varphi}$, and $\chi(W)$, we obtain the desired result.
\end{proof}

Theorem~\ref{thm:case1} states that in order to obtain an $(\varepsilon,\varepsilon/R)${-}solution of~\eqref{consensus_problem2}, when each function $f_i$ is strongly convex and smooth, the communication complexity is $O \big(\sqrt{\chi(W) {L}/{\mu}}\log (1/\varepsilon)\big)$.

\subsection{Sums of Strongly Convex and $M${-}Lipschitz Functions on a Bounded set}

In this subsection, we propose a distributed algorithm for sum of strongly convex functions that are Lipschitz on a bounded set to be specified later. Moreover, we show the convergence rates of the proposed algorithm.

We will build our results by using Nesterov smoothing \cite{nes05,dev12}. This approach has been previously used, for example in~\cite{nec08}. We use such approach to guarantee optimal dependency on the network structure. Furthermore, as pointed before, we remove the assumption of bounded variation on the admissible set. Particularly, we will use the following result

\begin{proposition}[\cub{Lemma 3 in \cite{gas16b}}]\label{prop:smooth}
	Consider a convex function $\varphi$, the strongly convex term $({\hat \mu }/{2})\|y\|_2^2$, and define
	$
	\hat\varphi(y) = \varphi(y) {+} ({\hat \mu }/{2})\|y\|_2^2.
	$
	Then, $\hat\varphi(y)$ is $\hat\mu${-}strongly convex. Moreover, if $\hat\mu \leq {\varepsilon}/{R^2}$, where $R = \|y^*\|\ag{_2}$ and assume that there exists $y_N$ such that
	$
	\hat\varphi(y_N) {-} \hat\varphi^* \leq {\varepsilon}/{2},
	$
	it holds that
	$
	\varphi(y_N) {-} \varphi^* \leq \varepsilon,
	$
	where $\varphi^*$ and $\hat\varphi^*$ are the optimal values of the function $\varphi$ and $\hat\varphi$ respectively. Moreover, if $\varphi$ is defined in~\eqref{dual_problem22}, then 
	$
	\nabla \hat\varphi(y)   = Ax^*(A^Ty) {+} \hat\mu y.
	$
\end{proposition}

\begin{algorithm}[t]
	\begin{algorithmic}[1]
		\STATE{All agents set $z_0^i = \tilde{z}_0^i =0 \in \mathbb{R}^n$, $q =\frac{{\varepsilon}/{(4R^2)}}{{\lambda_{\max}(W)}/{\mu} {+} {\varepsilon}/{(4R^2)}}$, $\alpha_0$ solves $\frac{\alpha_0^2 {-}q}{1{-}\alpha_0} = 1$ and $N$.}
		\STATE{For each agent $i$}
		\FOR{ $k=0,1,2,\cdots,N{-}1$ }
		\STATE{$x^*_i(\tilde z^i_k)  = \argmax\limits_{x_i} \left\lbrace \left\langle \tilde z^i_k,x_i \right\rangle {-}f_i(x_i) \right\rbrace$}
		\STATE{Share $x^{*}_i(\tilde z_k^i)$ with neighbors, i.e. $\{j \mid (i,j) \in E \}$.}
		\STATE{$z_{k{+}1}^i = \tilde z_k^i {-} \frac{1}{{\lambda_{\max}(W)}/{\mu} {+} {\varepsilon}/{(4R^2)}}\left(  \sum\limits_{j=1}^{m} W_{ij} x^{*}_j(\tilde z_k^j) {+} \frac{\varepsilon}{4R^2} \tilde z_k^i \right)$}
		\STATE{Compute $\alpha_{k{+}1}\in (0,1)$ from $\alpha_{k{+}1}^2 = (1 {-} \alpha_{k{+}1})\alpha_{k}^2 {+}q \alpha_{k{+}1}$ and set $\beta_k = \frac{\alpha_k(1{-}\alpha_k)}{\alpha_{k}^2 {+} \alpha_{k{+}1}}$}
		\STATE{{ $\tilde z_{k{+}1}^i  = z_{k{+}1}^i {+} \beta_k(z_{k{+}1}^i {-} z_k^i)$}}
		\ENDFOR
	\end{algorithmic}
	\caption{Distributed \textsc{FGM} for strongly convex and $M${-}Lipschitz problems}
\label{alg:case2}
\end{algorithm}

The main difference between Algorithm~\ref{alg:case1} and Algorithm~\ref{alg:case2} is that we have an additional term inside the parenthesis in Line~$6$. Moreover, the corresponding strong convexity constant of the dual function is $\varepsilon/R^2$, where $R$ can be an upper bound on the norm of optimal solution of the dual problem $y^*$, i.e., $\|y^*\|_2 \leq R$. Next, we state our main result regarding the number of iterations required by Algorithm~\ref{alg:case2} to reach an approximate solution of problem~\eqref{equiv_main_problem}.

\begin{theorem}\label{thm:case2}              
	Let $F(x)$ be dual friendly and Assumption \ref{assumptions}(b) hold. \ag{Moreover, assume $\|\nabla F(x^*)\|_2 \leq M$.} For any 
	$\varepsilon >0$, the output $x^*(z_N)$ of Algorithm~\ref{alg:case2} is an $(\varepsilon,\varepsilon /R)${-}solution of \eqref{consensus_problem2} for
	\begin{align*}
	N & \geq 2\sqrt{4 \chi(W)\frac{M^2 }{\mu \cdot \varepsilon} {+} 1} \log \left( 4 \chi(W)\frac{M^2 }{\mu \cdot \varepsilon} {+} 1  \right),
	\end{align*}
	where $\chi(W) = \lambda_{\max}{(W)}/ \lambda_{\min}^{{+}}{(W)}$.
\end{theorem}

\begin{proof}
	Initially, consider the regularized dual function $\hat \varphi$ with $\hat{\mu} = {\varepsilon}/{(4R^2)}$, which is $\mu_{\hat \varphi}${-}strongly convex with $\mu_{\hat \varphi} = {\varepsilon}/{(4R^2)}$, and $L_{\hat \varphi}${-}smooth with $L_{\hat \varphi}={\lambda_{\max}(W)}/{\mu}+{\varepsilon}/{(4R^2)}$. Thus, similarly as in \eqref{nest_dual} 
	\begin{align*}
	\hat \varphi(y_k){-} \hat \varphi^* & \leq L_{\hat \varphi} \hat R^2 \exp\left({-}k \sqrt{\frac{\mu_{\hat \varphi}}{L_{\hat \varphi}}}\right)\leq L_{\hat \varphi} R^2 \exp\left({-}k \sqrt{\frac{\mu_{\hat \varphi}}{L_{\hat \varphi}}}\right),
	\end{align*}
	where $\hat R = \| \hat y^* \|_2$, and $\hat y^*$ is the smallest norm solution of the regularized dual problem. Note that by definition $\hat R  = \|\hat y^*\|_2 \leq \|y^*\|_2 = R $. \ag{Note also, that $\|y_k - \hat{y}^*\|_2 \leq \|y_0 - \hat{y}^*\|_2 = \|\hat{y}^*\|_2$ (see formula (2.5) from \cite{dev12}). Therefore, $\|y_k\|_2 \leq  2\|\hat{y}^*\|_2 \leq 2\|y^*\|_2 $.} 
	
	
	Next, we provide a relation between the distance to optimality of the non{-}regularized primal problem and the regularized dual problem. Note that for any $y$ it holds that
	\begin{align*}
	\hat \varphi (y) {-} \hat \varphi^* &\geq \frac{\|\nabla \hat \varphi(y)\|_2^2}{2 L_{\hat \varphi}} = \frac{\|\nabla \varphi (y) {+} \hat \mu y\|_2^2}{2 L_{\hat \varphi}} \geq \frac{\hat \mu \left\langle y,\nabla \varphi(y) \right\rangle}{L_{\hat \varphi}}.
	\end{align*}
	Therefore, $
	\left\langle y,\nabla \varphi(y) \right\rangle  \leq ({L_{\hat \varphi}}/{\mu_{\hat \varphi}} )\left( \hat \varphi (y) {-} \hat \varphi^*\right) \leq ({4}/{\varepsilon})L_{\hat \varphi}^2 R^4 \exp\left({-}k \sqrt{{\mu_{\hat \varphi}}/{L_{\hat \varphi}}}\right)
	$.
	Consequently, if 
	$
	k~\geq~2\sqrt{{L_{\hat \varphi}}/{\mu_{\hat \varphi}}} \log \left( {2L_{\hat \varphi} R^{2}}/{\varepsilon} \right),
	$
	then $\left\langle y,\nabla \varphi(y) \right\rangle \leq \varepsilon$.
	
	Moreover, it follows from the definition of the regularized dual function that
	\begin{align*}
	\|\nabla \varphi(y_k)\|_2 & \leq \|\nabla \hat \varphi (y_k)\|_2 {+} \hat \mu \| y_k\|_2 \leq \sqrt{2 L_{\hat \varphi} (\hat \varphi (y) {-} \hat \varphi^*)} {+} \hat \mu \| y_k \|_2\\
	& \leq  \sqrt{2} L_{\hat \varphi} R \exp\left({-}\frac{k}{2} \sqrt{\frac{\mu_{\hat \varphi}}{L_{\hat \varphi}}}\right) {+}2 \hat \mu \hat R\leq  \sqrt{2} L_{\hat \varphi} R \exp\left({-}\frac{k}{2} \sqrt{\frac{\mu_{\hat \varphi}}{L_{\hat \varphi}}}\right) {+} \frac{\varepsilon}{2R}.
	\end{align*}
	
	Using the definition of the gradient of the dual function then we have that \newline
	\mbox{$
	\|\sqrt{W} x^*(\sqrt{W} \tilde y_k) \|_2  \leq {\varepsilon}/{R},
	$}
	for 
	$
	k  \geq 2\sqrt{{L_{\hat \varphi}}/{\mu_{\hat\varphi}}} \log \left({\sqrt{2} L_{\hat \varphi} R^2 }/{ \varepsilon} \right).
	$
	
	We conclude, from Lemma~\ref{lemma:sasha}, that we will have an $(\varepsilon,\varepsilon/R)$~solution of~\eqref{consensus_problem2} if
	\begin{align*}
	k & \geq 2\sqrt{\frac{L_{\hat \varphi}}{\mu_{\hat\varphi}}} \log \left( \max \left\lbrace \frac{2L_{\hat \varphi} R^{2}}{\varepsilon}  ,\frac{\sqrt{2} L_{\hat \varphi} R^2 }{ \varepsilon}  \right\rbrace  \right) \\
	& \geq  2\sqrt{\frac{{\lambda_{\max}(W)}/{\mu} {+} {\varepsilon}/{(4R^2)}}{ {\varepsilon}/{(4R^2)}}} \log \left(  \frac{2 R^{2} \left(\frac{\lambda_{\max}(W)}{\mu} {+} \frac{\varepsilon}{4R^2} \right) }{\varepsilon}  \right) \\
	& =  2\sqrt{\frac{4 R^2 \lambda_{\max}(W)}{\mu \cdot \varepsilon } {+} 1} \ \log \left(  \frac{4 R^2 \lambda_{\max}(W)}{\mu \cdot \varepsilon } {+} 1 \right).
	\end{align*}

	Now, we focus our attention to find a bound on the value $R$ such that we can provide an explicit dependency on the minimum non zero eigenvalue of the graph Laplacian. This will allow us to provide an explicit iteration complexity in terms of the condition number of the graph Laplacian. 
	
	Theorem $3$ in \cite{lan17} provides a bound that relates $R$ with the magnitude of the gradient of $F(x)$ at the optimal point $x = x^*$. Particularly, it is shown that
	\begin{align}\label{Soomin_bound}
	R^2 & = \|y^*\|_2^2 \leq {\|\nabla F(x^*)\|_2^2}/{\sigma_{\min}^{+}(A)}.
	\end{align} 
Thus, it follows from \ag{\cite{lan17}} that $
	R^2 \leq {M^2}/{\sigma_{\min}^{+}(A)}$, \ag{where $M = \|\nabla F(x^*)\|_2$.}
	Therefore to have an $(\varepsilon, \varepsilon/R)${-}solution it is necessary that
	\begin{align*}
k
& \geq 2\sqrt{4 \chi(W)\frac{M^2 }{\mu \cdot \varepsilon} {+} 1} \ \log \left( 4 \chi(W)\frac{M^2 }{\mu \cdot \varepsilon} {+} 1  \right). 
\end{align*}
		
\end{proof}

Theorem~\ref{thm:case2} states the communication complexity of Algorithm~\ref{alg:case2}. Particularly, the total number of communication rounds required by each agent to find an $(\varepsilon,\varepsilon /R)${-}solution of \eqref{consensus_problem2} can be bounded by $\tilde O\big(\sqrt{\chi(W) {M^2 }/{(\mu \cdot \varepsilon)} }\big)$.

\subsection{Sums of Smooth and Convex Functions}

In this subsection, we propose a distributed optimization algorithm to find a solution of the problem~\ref{equiv_main_problem}, when the function $F(x)$ is smooth, i.e., $F(x)$ has Lipschitz gradients. We follow the idea of regularization to induce strong convexity, this time on the primal problem. Moreover, we show the minimum number of iterations required to compute an approximate solution to the problem.

\begin{algorithm}[t]
	\begin{algorithmic}[1]
		\STATE{All agents set $z_0^i = \tilde{z}_0^i =0 \in \mathbb{R}^n$, $q =\frac{\varepsilon/R_x^2}{L{+}\varepsilon/R_x^2}\frac{\lambda^{{+}}_{\min}(W)}{\lambda_{\max}(W)}$, $\alpha_0$ solves $\frac{\alpha_0^2 {-}q}{1{-}\alpha_0} = 1$ and $N$.}
		\STATE{For each agent $i$}
		\FOR{ $k=0,1,2,\cdots,N{-}1$ }
		\STATE{Set $\hat x^*_i(\tilde z^i_k) = \argmax_{x_i} \{ \left\langle \tilde z^i_k,x_i \right\rangle  {-} f_i(x_i) {-} \frac{\varepsilon}{2 R^2_x}\|x_i {-}x^*_i(0)\|_2^2\}$} 
		\STATE{Share $\hat x^{*}_i(\tilde z_k^i)$ with neighbors, i.e. $\{j \mid (i,j) \in E \}$.}
		\STATE{{ $z_{k{+}1}^i = \tilde z_k^i {-} \frac{\varepsilon/R_x^2}{\lambda_{\max}(W)} \sum_{j=1}^{m} W_{ij} \hat x^{*}_j(\tilde z_k^j)$}}
		\STATE{Compute $\alpha_{k{+}1}\in (0,1)$ from $\alpha_{k{+}1}^2 = (1 {-} \alpha_{k{+}1})\alpha_{k}^2 {+}q \alpha_{k{+}1}$ and set $\beta_k = \frac{\alpha_k(1{-}\alpha_k)}{\alpha_{k}^2 {+} \alpha_{k{+}1}}$}
		\STATE{{ $\tilde z_{k{+}1}^i  = z_{k{+}1}^i {+} \beta_k(z_{k{+}1}^i {-} z_k^i)$}}
		\ENDFOR
	\end{algorithmic}
	\caption{Distributed \textsc{FGM} for smooth convex problems}
	\label{alg:case3}
\end{algorithm}

\begin{theorem}\label{thm:case3}
	Let $F(x)$ be dual friendly and Assumption \ref{assumptions}(c) hold. For any $\varepsilon >0$, the output $x^*(z_N)$ of Algorithm \ref{alg:case3} is an \mbox{$(\varepsilon,\varepsilon/R )${-}solution} of~\eqref{consensus_problem2} for
	\begin{align*}
	N & \geq 2\sqrt{\left( \frac{2LR_x^2}{\varepsilon} {+}1 \right) \chi(W)} \log \left(\frac{8\sqrt{2} \lambda_{\max}(W) R^2 R_x^2}{ \varepsilon^2}  \right). 
	\end{align*}
	where $\chi(W) = \lambda_{\max}{(W)}/ \lambda_{\min}^{{+}}{(W)}$ and $R_x =  \|x^*{-}x^*(0)\|_2$.
\end{theorem}

\begin{proof}
	Initially, consider the regularized problem 
	\begin{align}\label{prob:regularized_primal}
	\min_{\sqrt{W} x=0} \hat F(x) & \qquad 
	\text{where } \qquad \hat F(x) \triangleq F(x) {+} \frac{\varepsilon}{2 R^2_x}\|x {-} x^*(0)\|_2^2,
	\end{align}
	where $F(x)$ is defined in~\eqref{consensus_problem2}. The function $\hat F(x)$ is $\hat \mu${-}strongly convex with $\hat \mu = {\varepsilon}/{(2R_x^2)}$ and $\hat L${-}smooth with $\hat L = L {+} \hat \mu$. Given that the regularized primal function is strongly convex and smooth, we can use the results from Theorem~\ref{thm:case1}. Particularly, in order to have an $(\varepsilon/2,\varepsilon/(2R))${-}solution of problem~\eqref{prob:regularized_primal}, one can use Algorithm~\ref{alg:case1} with
	\begin{align*}
	N & \geq 2\sqrt{\frac{\hat L}{\hat \mu}\chi(W)} \log \left(\frac{4\sqrt{2} \lambda_{\max}(W) R^2}{ \hat  \mu \cdot \varepsilon}  \right)\\
	& = 2\sqrt{\left( \frac{2LR_x^2}{\varepsilon} {+}1 \right) \chi(W)} \log \left(\frac{8\sqrt{2} \lambda_{\max}(W) R^2 R_x^2}{ \varepsilon^2}  \right).
	\end{align*}	
	Having an $(\varepsilon/2,\varepsilon/(2R))${-}solution of problem~\eqref{prob:regularized_primal}, guarantees that $\hat x^*_N$ is an approximate \mbox{$(\varepsilon,\varepsilon/R)${-}solution} of problem~\eqref{consensus_problem2}, and the desired result follows.
\end{proof}

Theorem~\ref{thm:case3} states the communication complexity of Algorithm~\ref{alg:case3}. Particularly, the total number of communication rounds required by each agent to find an \newline $(\varepsilon,\varepsilon /R)${-}solution of \eqref{consensus_problem2} can be bounded by $\tilde O\big(\sqrt{\chi(W) {LR_x^2 }/{\varepsilon} }\big)$.

\subsection{Sums of Convex and $M${-}Lipschitz Functions}

In this subsection, we present the distributed algorithm for optimization of convex function when no strong convexity or smoothness is guaranteed. The main idea is to use regularization both in the primal and the dual problem. Therefore, we can build our algorithm and its analysis from the results in Theorem~\ref{thm:case2} and Theorem~\ref{thm:case3}. Next, we present the proposed algorithm and their convergence analysis.

\begin{algorithm}[t]
	\begin{algorithmic}[1]
		\STATE{All agents set $z_0^i = \tilde{z}_0^i =0 \in \mathbb{R}^n$, $q =\frac{{\varepsilon}/{(4R^2)}}{ {\lambda_{\max}(W)}/{(\varepsilon / R_x^2)}{+} {\varepsilon}/{(4R^2)}}$, $\alpha_0$ solves $\frac{\alpha_0^2 {-}q}{1{-}\alpha_0} = 1$ and $N$.}
		\STATE{For each agent $i$}
		\FOR{ $k=0,1,2,\cdots,N{-}1$ }
		\STATE{Set $\hat x^*_i(\tilde z^i_k) = \argmax_{x_i} \{ \left\langle \tilde z^i_k,x_i \right\rangle  {-} f_i(x_i) {-} \frac{\varepsilon}{2 R^2_x}\|x_i {-}x^*_i(0)\|_2^2\}$} 
		\STATE{Share $\hat x^{*}_i(\tilde z_k^i)$ with neighbors, i.e. $\{j \mid (i,j) \in E \}$.}
		\STATE{$z_{k{+}1}^i = \tilde z_k^i {-} \frac{1}{ {\lambda_{\max}(W)}/{(\varepsilon / R_x^2)}{+} {\varepsilon}/{(4R^2)}}\left(  \sum\limits_{j=1}^{m} W_{ij} \hat  x^{*}_j(\tilde z_k^j) {+} \frac{\varepsilon}{4R^2} \tilde z_k^i \right)$}
		\STATE{Compute $\alpha_{k{+}1}\in (0,1)$ from $\alpha_{k{+}1}^2 = (1 {-} \alpha_{k{+}1})\alpha_{k}^2 {+}q \alpha_{k{+}1}$ and set $\beta_k = \frac{\alpha_k(1{-}\alpha_k)}{\alpha_{k}^2 {+} \alpha_{k{+}1}}$}
		\STATE{{ $\tilde z_{k{+}1}^i  = z_{k{+}1}^i {+} \beta_k(z_{k{+}1}^i {-} z_k^i)$}}
		\ENDFOR
	\end{algorithmic}
	\caption{Distributed \textsc{FGM} for $M${-}Lipschitz functions}
	\label{alg:case4}
\end{algorithm}

\begin{theorem}\label{thm:case4}
	Let $F(x)$ be dual friendly and Assumption \ref{assumptions}(d) hold. For any $\varepsilon >0$, the output $x^*(z_N)$ of Algorithm~\ref{alg:case4} is an \mbox{$(\varepsilon,\varepsilon/R )${-}solution} of~\eqref{consensus_problem2} for
	\begin{align*}
	N & \geq 2\sqrt{16 \chi(W)\frac{M^2  R^2_x}{\varepsilon^2} {+} 1} \log \left( 16 \chi(W)\frac{M^2 R^2_x}{  \varepsilon^2} {+} 1  \right),
	\end{align*}
	where $\chi(W) = \lambda_{\max}{(W)}/ \lambda_{\min}^{{+}}{(W)}$, $R = \|y^*\|_2$, and $R_x =  \|x^*{-}x^*(0)\|_2$.
\end{theorem}

\begin{proof}
	Consider again, as in Theorem~\ref{thm:case3}, the regularized problem~\eqref{prob:regularized_primal} where $F(x)$ is defined in~\eqref{consensus_problem2}. The function $\hat F(x)$ is $\hat \mu${-}strongly convex with $\hat \mu = {\varepsilon}/{(2R_x^2)}$. However, we have assumed now that $F(x)$ is not smooth. Nevertheless, from Theorem~\ref{thm:case2}, we have that Algorithm~\ref{alg:case2} will generate an $(\varepsilon/2,\varepsilon /(2R))${-}solution of \eqref{prob:regularized_primal}, namely $x^*_N$, for
	\begin{align*}
	N & \geq 2\sqrt{8 \chi(W)\frac{M^2 }{\hat \mu \cdot \varepsilon} {+} 1} \log \left( 8 \chi(W)\frac{M^2 }{\hat \mu \cdot \varepsilon} {+} 1  \right) \\
		& = 2\sqrt{16 \chi(W)\frac{M^2  R^2_x}{\varepsilon^2} {+} 1} \log \left( 16 \chi(W)\frac{M^2 R^2_x}{  \varepsilon^2} {+} 1  \right).
	\end{align*}
	Therefore, by Proposition~\ref{prop:smooth}, $x^*(z_N)$ is an $(\varepsilon,\varepsilon /R)${-}solution for problem \eqref{consensus_problem2}.
\end{proof}
Theorem~\ref{thm:case4} states the communication complexity of Algorithm~\ref{alg:case4}. Particularly, the total number of communication rounds required by each agent to find an $(\varepsilon,\varepsilon /R)${-}solution of \eqref{consensus_problem2} can be bounded by $\tilde O\big(\sqrt{\chi(W) {R_x^2 M^2 }/{\varepsilon^2} }\big)$.
	
	\section{Non-dual Friendly Functions: Algorithms and Iteration Complexity Analysis}\label{sec:non_dual}
	
	The results in Section~\ref{sec:main} assume $F(x)$ is \textit{dual-friendly}. In this section, we explore the case when no exact solution to the dual problem is available. We state the algorithms and their convergence rates for the distributed optimization of sums of non-dual friendly convex functions.

We will build on the results in~\cite{Devolder2013,Devolder2014} on the analysis of first-order methods with inexact oracle, and provide a set of distributed algorithms and their respective iterations complexities. Initially for completeness, we recall the definition of an inexact oracle for \cub{a} smooth strongly convex function, and the corresponding iteration complexity of \textsc{FGM} with an inexact oracle.  

\begin{definition}[Definition $1$ in \cite{Devolder2013}]
	Let function $f$ be convex on a convex set $Q$. We say that it is equipped with a first-order $(\delta,L,\mu)$-oracle if for any $y\in Q$ we can compute a pair $ (f_{\delta,L,\mu}(y),g_{\delta,L,\mu}(y)) \in \mathbb{R} \times \mathbb{R}^n$ such that
	\begin{align*}
	\frac{\mu}{2} \|x-y\|_2^2 \leq f(x) - (f_{\delta,L,\mu}(y) + \langle g_{\delta,L,\mu}(y), x-y \rangle ) \leq \frac{L}{2} \|x-y\|_2^2 + \delta,
	\end{align*}
	for all $x\in Q$ where $\delta \geq 0$ and $L \geq \mu \geq 0$.
\end{definition}

\begin{theorem}[Theorem~$7$ in \cite{Devolder2013}]\label{thm:devoler}
	Nesterov's fast gradient method applied to a function $f$ endowed with a $(\delta,L,\mu)$-oracle generates a sequence $\{y_k\}_{k>1}$ satisfying:
	\begin{align*}
	f(y_k) - f^* & \leq L R^2 \exp \left(-\frac{k}{2} \sqrt{\frac{\mu}{L}}\right) + \left(1 + \sqrt{\frac{L}{\mu}}\right)\delta,
	\end{align*}
	where $R = \|y^*\|_2.$
\end{theorem}

Now, we recall an auxiliary result that shows that an approximate solution to the auxiliary inner maximization problem~\eqref{eq:xstary} defines a $(\delta,L,\mu)$-oracle. Furthermore, we describe the distributed algorithms and their iterations complexities when we remove the assumption of the function $F$ being dual friendly.
\begin{theorem}[Theorem $2$ in \cite{Devolder2013}]\label{thm:endowed}
	Assume that $f$ is $\mu$-strongly convex and $L$-smooth. \cu{For a given $y\in R^n$,} assume that instead of computing $x^*(A^Ty)$, {i.e.,} the unique solution of the subproblem~\eqref{eq:xstary}, \cub{for an $\xi>0$,} we compute $w(A^Ty)$ such that:
	\begin{align}\label{bound:psi}
	\Psi(y,x^*(A^Ty)) - \Psi(y,w(A^Ty))  \leq \xi.
	\end{align}  
	Then,
	\begin{align*}
	(\Psi(y,w(A^Ty)) - \xi = f(w(A^Ty)) + \langle A w(A^Ty) , y \rangle -\xi , Aw(A^Ty) )
	\end{align*}
	is a $(\delta,L_{\varphi,\delta},\mu_{\varphi,\delta})$-oracle for $\varphi$ with $\delta = 3\xi$, $L_{\varphi,\delta} = 2L_{\varphi}$ and $\mu_{\varphi,\delta}  = ({1}/{2}) \mu_{\varphi}$.
\end{theorem}

\subsection{Sum of Non-dual Friendly Strongly Convex and Smooth Functions}

In this subsection, we introduce a distributed algorithm for the minimization of sums of strongly convex and smooth functions removing the assumption of dual friendliness. Moreover, we provide its iteration complexity.


\begin{algorithm}[t]
	\begin{algorithmic}[1]
		\STATE{All agents set $z_0^i = \tilde{z}_0^i =0 \in \mathbb{R}^n$, $\tilde q = \frac{\mu}{L}$, $q =\tilde q \frac{\lambda^{+}_{\min}(W)}{\lambda_{\max}(W)}$, $\alpha_0$ solves $\frac{\alpha_0^2 -q}{1-\alpha_0} = 1$, $\tilde \alpha_0$ solves $\frac{\tilde\alpha_0^2 -\tilde q}{1-\tilde\alpha_0} = 1$, $T$ and $N$.}
		\STATE{For each agent $i$}
		\FOR{ $k=0,1,2,\cdots,N-1$ }
		\STATE{ $w^i_0 = \tilde{w}^i_0 = 0 \in \mathbb{R}^n$}
			\FOR{ $t=0,1,2,\cdots,T-1$}
		\STATE{$w^i_{t+1} = \tilde{w}^i_t + \frac{1}{L} (\tilde z^i_k -\nabla f_i (\tilde{w}^i_t)) $}
		\STATE{Compute $\tilde \alpha_{t+1}\in (0,1)$ from $\tilde \alpha_{t+1}^2 = (1 -  \tilde\alpha_{t+1}) \tilde \alpha_{t}^2 + \tilde q \tilde \alpha_{t+1}$ and set $\tilde \beta_t = \frac{\tilde \alpha_t(1-\tilde\alpha_t)}{\tilde \alpha_{t}^2 + \tilde  \alpha_{t+1}}$}
		\STATE{{ $\tilde w_{t+1}^i  = w_{t+1}^i + \tilde \beta_t(w_{t+1}^i - w_t^i)$}}
		\ENDFOR
		\STATE{Share $w^i_T$ with neighbors, i.e. $\{j \mid (i,j) \in E \}$.}
		\STATE{{ $z_{k+1}^i = \tilde z_k^i - \frac{\mu}{\lambda_{\max}(W)} \sum_{j=1}^{m} W_{ij} w^j_T$}}
		\STATE{Compute $\alpha_{k+1}\in (0,1)$ from $\alpha_{k+1}^2 = (1 - \alpha_{k+1})\alpha_{k}^2 +q \alpha_{k+1}$ and set $\beta_k = \frac{\alpha_k(1-\alpha_k)}{\alpha_{k}^2 + \alpha_{k+1}}$}
		\STATE{{ $\tilde z_{k+1}^i  = z_{k+1}^i + \beta_k(z_{k+1}^i - z_k^i)$}}
		\ENDFOR
	\end{algorithmic}
	\caption{Distributed \textsc{FGM} for non-dual friendly strongly convex and smooth problems}
\label{alg:no_friend:case1}
\end{algorithm}

\cu{
\begin{remark}
	Algorithm~\ref{alg:no_friend:case1} has two loops: one inner loop that runs for $T$ iterations, and \cub{seeks} to compute an approximate solution to the dual auxiliary problem; and one outer loop that runs for $N$ iterations that applies FGM on the dual problem. Note that Lines $5-7$ compute an approximate gradient of the dual function, and Lines $11-13$ execute FGM with the inexact gradient computed in the inner loop.
\end{remark}
}

\begin{theorem}\label{thm:no_friend:case1}
	Let $F(x)$ be a function such that Assumption~\ref{assumptions}(a) hold. For any $\varepsilon >0$, the output $x^*(z_N)$ of Algorithm~\ref{alg:no_friend:case1} is an $(\varepsilon,\varepsilon/R )$-solution of~\eqref{consensus_problem2} for 
	\begin{align*}
	N & \geq 8\sqrt{\frac{L}{\mu}\chi(W)}\log \left(\frac{2\sqrt{2}\lambda_{\max}(W) R^2}{\mu\cdot\varepsilon}\right)
	\end{align*}
	and
	\begin{align*}
	T & \geq \sqrt{\frac{L}{\mu}} \log \left(\frac{6LR^2 R_w^2}{\varepsilon^2}\sqrt{\frac{L}{\mu}\chi(W)}\right),
	\end{align*}
	where $R = \|y^*\|_2$, $R_x = \|x^*-x^*(0)\|_2$, $R_w = R_x + \|x^*\|_2$ and $\chi(W)={\lambda_{\max}{(W)}}/{\lambda_{\min}^{+}{(W)}}$. 
\end{theorem}

\begin{proof}
	Lines $5-7$ in Algorithm~\ref{alg:no_friend:case1} are the \textsc{FGM} on the inner problem~\eqref{eq:xstary}. \cu{Therefore, it follows from classical analysis of FGM~\cite{nes13} that} $\Psi(y,x^*(A^Ty)) - \Psi(y,w_T(A^Ty))  \leq \xi$ for
	$
		T \geq \sqrt{{L}/{\mu}} \log \left({L R_w^2}/{\xi}\right)
	$.
	Note that at the beginning of iteration $k$, $R_w = \|w_0 - w^*\|$, $w_0 = 0$, and $w^* = x^*(\tilde z_k)$. Therefore, $R_w = \|x^*(\tilde z_k)\|\ag{_2} \leq \|x^*(\tilde z_k) - x^*\|_2 + \|x^*\|_2 \leq \|x^*(0) - x^*\|\ag{_2}+ \|x^*\|_2 = R_x + \|x^*\|\ag{_2}$.
	
	Moreover, Theorem~\ref{thm:endowed} shows us that we have endowed the function $\varphi$ with an $(3\xi,2L_{\varphi},\frac{1}{2} \mu_{\varphi})$-oracle. Thus, from Theorem~\ref{thm:devoler} it holds that
	\begin{align*}
	\varphi(y_k) - \varphi^* & \leq L_\varphi R^2 \exp \left(-\frac{k}{4} \sqrt{\frac{\mu_\varphi}{L_\varphi}}\right) + \left(1 + 2\sqrt{\frac{L_\varphi}{\mu_\varphi}}\right)3\xi.
	\end{align*}
	
	Now, recall that
	\begin{align*}
	\|\sqrt{W} x^*(\sqrt{W} \tilde y_k)\|_2^2 & \leq  2 L_{\varphi} \left(\varphi(y_k) - \varphi^*\right),\\
	 &\leq 2 L_{\varphi}^2R^2 \left(\exp \left(-\frac{k}{4} \sqrt{\frac{\mu_\varphi}{L_\varphi}}\right) + \left(1 + 2\sqrt{\frac{L_\varphi}{\mu_\varphi}}\right)3\xi\right).
	\end{align*}
	Therefore, in order to have $\|\sqrt{W} x^*(\sqrt{W} \tilde y_k)\|_2 \leq \varepsilon/R$ it is necessary that
	\begin{align*}
	N & \geq 8\sqrt{\frac{L_\varphi}{\mu_\varphi}}\log \left(\frac{\sqrt{6}L_\varphi R^2}{\varepsilon}\right) \quad \text{and} \quad \xi \leq \frac{\varepsilon^2}{6R^2}\sqrt{\frac{\mu_{\varphi}}{L_\varphi}}.
	\end{align*}
	Moreover,
	\begin{align*}
	|\langle y_k,\sqrt{W} x^*(\sqrt{W} y_k) \rangle|^2 & \leq 4 R^2 \|\sqrt{W} x^*(\sqrt{W} y_k)\|_2^2,\\
	& \leq 8 R^4  L_{\varphi}^2\left(  \exp \left(-\frac{k}{4} \sqrt{\frac{\mu_\varphi}{L_\varphi}}\right) + \left(1 + 2\sqrt{\frac{L_\varphi}{\mu_\varphi}}\right)3\xi\right) .
	\end{align*}
	Therefore, in order to have $f(z_N) - f^* \leq \varepsilon$ it is necessary that
	\begin{align*}
	N & \geq 8\sqrt{\frac{L_\varphi}{\mu_\varphi}}\log \left(\frac{\sqrt{8}L_\varphi R^2}{\varepsilon}\right) \quad \text{and} \quad \xi \leq \frac{\varepsilon^2}{6}\sqrt{\frac{\mu_{\varphi}}{L_\varphi}}.
	\end{align*}
	Finally, we can conclude that to obtain an $(\varepsilon,\varepsilon/R)$-solution of~\eqref{consensus_problem2} we require
	\begin{align*}
	 N & \geq 8\sqrt{\frac{L_\varphi}{\mu_\varphi}}\log \left(\frac{2\sqrt{2}L_\varphi R^2}{\varepsilon}\right) \quad \text{and} \quad T \geq \sqrt{\frac{L}{\mu}} \log \left(\frac{6LR^2 R_w^2}{\varepsilon^2}\sqrt{\frac{L_\varphi}{\mu_\varphi}}\right).
	\end{align*}
	
	The desired result follows from the definitions of $L_\varphi$ and $\mu_\varphi$.
\end{proof}

Theorem~\ref{thm:no_friend:case1} shows that if no-dual solution is explicitly available for~\eqref{eq:xstary}, then one can use \textsc{FGM} to find an approximate solution. This approximate solution is itself an inexact oracle. Particularly, the number of total number of communication rounds required by each agent to find an $(\varepsilon,\varepsilon/R)$-solution of~\eqref{consensus_problem2} can be bounded by ${O}\big(\sqrt{\chi(W){L}/{\mu}} \log(1/\varepsilon)\big)$. Moreover, at each communication round the number of local oracle calls for each agent can be bounded by ${O}\big(\sqrt{{L}/{\mu}} \log(1/\varepsilon)\big)$. Unfortunately, the total number of oracle calls for each agent at all communication rounds is bounded by ${O}\big(({L}/{\mu})\sqrt{\chi(W)}\log^2(1/\varepsilon)\big)$, which is not optimal compared with their centralized \textsc{FGM} where the number of oracle calls of the function $F$ is bounded by $O\big(\sqrt{{L}/{\mu}}\log(1/\varepsilon)\big)$. However, in the centralized case, one oracle call corresponds to the gradient computation of $F(x)$ which is composed by $m$ functions $f_i$ for $1\leq i \leq m$, whereas in the distributed case, local oracle calls are computed in parallel by all agents at the same time. Therefore, one can argue that if the number of agent $m$ is of the order of $\sqrt{{L}/{\mu}}$, then provided estimates are optimal given that the oracle calls of all agents in the network are done in parallel.

\subsection{Sums of Non-dual Friendly Smooth Convex Functions}

In this subsection, we propose a distributed algorithm for the distributed minimization of sums of non-dual friendly smooth convex functions and provide its iteration complexity.


\begin{algorithm}[t]
	\begin{algorithmic}[1]
		\STATE{All agents set $z_0^i = \tilde{z}_0^i =0 \in \mathbb{R}^n$, $\tilde q = \frac{{\varepsilon}/{R^2_x}}{L+{\varepsilon}/{R^2_x}}$, $q =\tilde q \frac{\lambda^{+}_{\min}(W)}{\lambda_{\max}(W)}$, $\alpha_0$ solves $\frac{\alpha_0^2 -q}{1-\alpha_0} = 1$, $\tilde \alpha_0$ solves $\frac{\tilde\alpha_0^2 -\tilde q}{1-\tilde\alpha_0} = 1$, $T$ and $N$.}
		\STATE{For each agent $i$}
		\FOR{ $k=0,1,2,\cdots,N-1$ }
		\STATE{ $w^i_0 = \tilde{w}^i_0 = 0 \in \mathbb{R}^n$}
		\FOR{ $t=0,1,2,\cdots,T-1$}
		\STATE{$w^i_{t+1} = \tilde{w}^i_t + \frac{1}{L+{\varepsilon}/{R^2_x}} \left( \tilde z^i_k -\nabla f_i (\tilde{w}^i_t) - \frac{\varepsilon}{R^2_x}\left( \tilde w_t^i - x^*_i(0)\right) \right)  $}
		\STATE{Compute $\tilde \alpha_{t+1}\in (0,1)$ from $\tilde \alpha_{t+1}^2 = (1 -  \tilde\alpha_{t+1}) \tilde \alpha_{t}^2 + \tilde q \tilde \alpha_{t+1}$ and set $\tilde \beta_t = \frac{\tilde \alpha_t(1-\tilde\alpha_t)}{\tilde \alpha_{t}^2 + \tilde  \alpha_{t+1}}$}
		\STATE{{ $\tilde w_{t+1}^i  = w_{t+1}^i + \tilde \beta_t(w_{t+1}^i - w_t^i)$}}
		\ENDFOR
		\STATE{Share $w^i_T$ with neighbors, i.e. $\{j \mid (i,j) \in E \}$.}
		\STATE{{ $z_{k+1}^i = \tilde z_k^i - \frac{{\varepsilon}/{R^2_x}}{\lambda_{\max}(W)} \sum_{j=1}^{m} W_{ij} w^j_T$}}
		\STATE{Compute $\alpha_{k+1}\in (0,1)$ from $\alpha_{k+1}^2 = (1 - \alpha_{k+1})\alpha_{k}^2 +q \alpha_{k+1}$ and set $\beta_k = \frac{\alpha_k(1-\alpha_k)}{\alpha_{k}^2 + \alpha_{k+1}}$}
		\STATE{{ $\tilde z_{k+1}^i  = z_{k+1}^i + \beta_k(z_{k+1}^i - z_k^i)$}}
		\ENDFOR
	\end{algorithmic}
	\caption{Distributed \textsc{FGM} for non-dual friendly smooth problems}
	\label{alg:no_friend:case2}
\end{algorithm}

\begin{theorem}\label{thm:no_friend:case2}
	Let $F(x)$ be a function such that Assumption~\ref{assumptions}(c) hold. For any $\varepsilon >0$, the output $x^*(z_N)$ of Algorithm~\ref{alg:no_friend:case2} is an $(\varepsilon,\varepsilon/R )$-solution of~\eqref{consensus_problem2} for 
	\begin{align*}
	N & \geq 8\sqrt{ \left( \frac{2LR_x^2}{\varepsilon} + 1 \right)  \chi(W)}\log \left(\frac{8\sqrt{2}\lambda_{\max}(W) R^2_x R^2}{ \varepsilon^2}\right),
	\end{align*}
	and
	\begin{align*}
	T & \geq \sqrt{\frac{2LR_x^2}{\varepsilon} + 1} \log \left(\frac{2\sqrt{6} R^2 R_w^2}{\varepsilon} \left(\frac{L}{\varepsilon} + \frac{1}{2R_x^2} \right) \sqrt{ \left( \frac{2LR_x^2}{\varepsilon} + 1 \right)\chi(W)}\right),
	\end{align*}
	where $R = \|y^*\|_2$, $R_x = \|x^*-x^*(0)\|_2 $, $R_w = R_x + \|x^*\|_2$, and $\chi(W)~=~{\lambda_{\max}{(W)}}/{\lambda_{\min}^{+}{(W)}}$.
\end{theorem}

\begin{proof}
	Similarly as in the dual friendly case, we consider the regularized primal problem,
	\cu{	\begin{align}\label{prob:regularized_primal2}
	\min_{\sqrt{W} x=0} \hat F(x) & \qquad 
	\text{where } \qquad \hat F(x) \triangleq F(x) + \frac{\varepsilon}{2 R^2_x}\|x - x^*(0)\|_2^2,
	\end{align}}
	 Therefore, the auxiliary inner maximization problem seeks to maximize an \mbox{$\hat \mu$-strongly} convex function, with $\hat \mu = {\varepsilon}/{(2R_x^2)}$, that is also $\hat L$-smooth, with $\hat L = L + \hat \mu$. It follows from Theorem~\ref{thm:no_friend:case1} that Algorithm~\ref{alg:no_friend:case2} will generate obtain an $(\varepsilon/2, \varepsilon/(2R))$-solution for
	\begin{align*}
	N & \geq 8\sqrt{\frac{\hat L}{\hat \mu}\chi(W)}\log \left(\frac{4\sqrt{2}\lambda_{\max}(W) R^2}{\hat \mu\cdot\varepsilon}\right)\\
	& = 8\sqrt{\frac{L + \frac{\varepsilon}{2R_x^2}}{\frac{\varepsilon}{2R_x^2}}\chi(W)}\log \left(\frac{4\sqrt{2}\lambda_{\max}(W) R^2}{\frac{\varepsilon}{2R_x^2} \cdot\varepsilon}\right) \\
	& =  8\sqrt{ \left( \frac{2LR_x^2}{\varepsilon} + 1 \right)  \chi(W)}\log \left(\frac{8\sqrt{2}\lambda_{\max}(W) R^2_x R^2}{ \varepsilon^2}\right),
	\end{align*}
	and
	\begin{align*}
	T & \geq \sqrt{\frac{\hat L}{\hat \mu}} \log \left(\frac{2\sqrt{6}\hat L R^2 R_w^2}{\varepsilon^2}\sqrt{\frac{\hat L}{\hat \mu}\chi(W)}\right)\\
	& = \sqrt{\frac{2LR_x^2}{\varepsilon} + 1} \log \left(\frac{2\sqrt{6} R^2 R_w^2}{\varepsilon} \left(\frac{L}{\varepsilon} + \frac{1}{2R_x^2} \right) \sqrt{ \left( \frac{2LR_x^2}{\varepsilon} + 1 \right)\chi(W)}\right),
	\end{align*}
	and the desired result follows.
	\qed
\end{proof}

Theorem~\ref{thm:no_friend:case2} shows that the total number of communication rounds required by each agent to find an $(\varepsilon,\varepsilon/R)$-solution of~\eqref{consensus_problem2}, when te functions are not strongly convex, can be bounded by $\tilde{O}\big(\sqrt{\chi(W){(LR_x^2)}/{\varepsilon}}\big)$. Moreover, at each communication round the number of local oracle calls for each agent can be bounded by $\tilde{O}\big(\sqrt{{(LR_x^2)}/{\varepsilon}}\big)$. Similarly as in Theorem~\ref{thm:no_friend:case1}, the total number or oracle calls of each agent can be bounded by $O\big(({LR_x^2}/{\varepsilon})\sqrt{\chi(W)}\big)$.

\subsection{Distributed Optimization of Sums of Non-smooth Functions}

In this subsection, we present an approach for developing distributed algorithms for non-smooth optimization, i.e., either Assumption~\ref{assumptions}(b) or Assumption~\ref{assumptions}(d) hold. These scenarios has been recently studied in~\cite{lan17}, where similar convergence rates have been derived. However, our particular selection of $\sqrt{W}$ instead of $W$ allows for a better dependency in terms of the graph condition number. 

{
Initially, following \cite{gasnikov2017universal}, consider a convex function $f$ that is also $M$-Lipschitz, i.e., Assumption~\ref{assumptions}(d), and consider the auxiliary problem
\begin{align*}
 \min_{x} F_\varepsilon(x) = f(x) + \frac{R^2}{\varepsilon} \cu{\|Ax \|_2^2.}
\end{align*}
Note that if one is able to find an $\bar x$ such that
\begin{align*}
F_\varepsilon(\cub{\bar x}) - \min_x F_\varepsilon (x) \leq \varepsilon,
\end{align*}
then it also holds that
\begin{align*}
f(\bar{x}) - \min_{Ax=0} f(x)  \leq \varepsilon, \ \ \text{ and } \ \ \|A \bar{x}\|_2 \leq \frac{2\varepsilon}{R}
\end{align*}
Note that the term $({R^2}/{\varepsilon})\|Ax \|_2^2$ is $L_\varepsilon$-smooth, with $L_\varepsilon = {\lambda_{\max}(A^TA)R^2}/{\varepsilon}$, and $f$ is $M$-Lipschitz~\cite{gorbunov2019optimal}. For this class of composite problems, one can use the accelerated gradient sliding method proposed in~\cite{lan16}. As a result, it follows from Corollary~$2$ in~\cite{lan16}, that in order to find an $(\varepsilon,\varepsilon/R)$-solution for problem~\eqref{consensus_problem2}, the total number of communication rounds and oracle calls can be bounded by 
\begin{align*}
O\left(\sqrt{\frac{ L_\varepsilon R_x^2}{\varepsilon}}\right) = O\left(\sqrt{\frac{ M^2 R_x^2}{\varepsilon^2} \chi(W)}\right),
\end{align*}
and
\begin{align*}
O\left(\frac{M^2R_x^2}{\varepsilon^2} + \sqrt{\frac{L_\varepsilon R_x^2}{\varepsilon}}\right) = O\left(\frac{M^2R_x^2}{\varepsilon^2} + \sqrt{\frac{ M^2 R_x^2}{\varepsilon^2} \chi(W)}\right),
\end{align*}
respectively.
}

Similarly, if we additionally assume that the function $f$ is $\mu$-strongly convex, i.e., Assumption~\ref{assumptions}(b). Then, from Theorem~$3$ in~\cite{lan16} it follows that the total number of oracle calls for $F_\varepsilon$ and $f$  required by the multi-phase gradient sliding algorithm to find an $(\varepsilon)$-solution for problem~~\eqref{consensus_problem2} can be bounded by
\begin{align*}
	O\left(\sqrt{\frac{L_\varepsilon}{\mu}} \log \left(\frac{R_x}{\varepsilon}\right)\right) = O\left(\sqrt{\frac{M^2}{\varepsilon \mu} \chi(W)} \log \left(\frac{\mu R_x^2}{\varepsilon}\right)\right)
\end{align*}
and
\begin{align*}
	O\left(\frac{M^2}{\varepsilon\mu} + \sqrt{\frac{L_\varepsilon}{\mu}} \log \left(\frac{R_x}{\varepsilon}\right) \right) = O\left(\frac{M^2}{\varepsilon\mu}  + \sqrt{\frac{M^2}{\varepsilon \mu} \chi(W)} \log \left(\frac{\mu R_x^2}{\varepsilon}\right)\right),
\end{align*}
respectively.

If follows from~\cite{lan16} that the above estimates are optimal up to logarithmic factors. Moreover, one can extend these results to stochastic optimization problems and the estimations will not change \cite{lan17}.
	
	\section{Improving on the Dependence of the Strong Convexity Parameter: Computation-Communication Trade-Off} \label{sec:improve_bound}
	
	Considering the general problem in~\eqref{main_problem}, the condition number ${L}/{\mu}$ can be large if one of the $\mu_i$ is small or even zero. It follows from Section~\ref{sec:non_dual} that the iteration complexity of the proposed algorithms can be very large, even if only one of the functions has a small strong convexity. In this section, we propose a reformulation of the original Problem~\eqref{main_problem} such that the dependency on the individual strong convexity constants can be improved. However, we will see that the improvement on the dependency of the condition number of the function $F$ comes at a price in terms of the communication rounds. 

Consider the following problem: 
\begin{align}\label{better_cond}
\min_{\sqrt{W} x=0} F_\alpha (x) & = F(x) + \frac{\alpha}{2}\left\langle x,Wx\right\rangle = \sum\limits_{i=1}^{m}f_i(x_i) + \frac{\alpha}{2}\left\langle x,Wx\right\rangle.
\end{align}
A solution to~\eqref{better_cond} is clearly a \cub{solution to~\eqref{main_problem}}.

The function $F_\alpha$ is \mbox{$\mu_\alpha$-strongly} convex with $\mu_\alpha = \min \left\lbrace \sum_{i=1}^{m}\mu_i,\alpha \lambda_{\min}(W)\right\rbrace $ and has \mbox{$L_\alpha $-Lipschitz} continuous gradients with $L_\alpha = L + \alpha \lambda_{\max}(W)$. Choose $\alpha = {\sum_{i=1}^{m} \mu_i }/{\lambda_{\min}(W)}$ and the function $F_\alpha$ will have a condition number
\begin{align*}
\frac{L_\alpha}{\mu_\alpha} = \frac{\max_i L_i}{\sum_{i=1}^{m} \mu_i} + \frac{\lambda_{\max}(W)}{\lambda_{\min}(W)} = \frac{\max_i L_i}{\sum_{i=1}^{m} \mu_i} + \chi(W).
\end{align*}

Unfortunately, the structure of the function $F_\alpha$ does not allow a decentralized computation of a solution for the inner problem~\eqref{eq:xstary}, i.e., each agent can compute the solution $x^*_i$ using local information only. Nevertheless, the additional term in~\eqref{better_cond} has a gradient with a network structure and can be computed in a distributed manner using information shared from the neighbors of each agent. Particularly, \cub{consider} the auxiliary dual problem
\begin{align}\label{eq:auxiliary_augmented}
\varphi_{\alpha}(y) & = \max_x \Psi_\alpha(x,y) \quad \text{where} \quad \Psi_\alpha(x,y) =\langle x, \sqrt{W}^Ty \rangle - F(x) - \frac{\alpha}{2}\left\langle x,Wx\right\rangle.
\end{align}
Then, we have that $
\nabla_x \Psi_\alpha(x,y) = \sqrt{W}^Ty - \nabla F(x) - \alpha Wx$. \ga{$\sqrt{W}^T = \sqrt{W}$...}

In this case, we can use the \textsc{FGM} to obtain an approximate solution to the inner problem using Nesterov fast gradient method. In Algorithm~\ref{alg:improve1}, we propose a modification of  Algorithm~\ref{alg:no_friend:case1} to take into account the new structure for the solution of the inner auxiliary problem. 
%

\begin{algorithm}[t]
	\caption{Augmented Distributed \textsc{FGM} for strongly convex and smooth problems}
	\begin{algorithmic}[1]
		\STATE{All agents set $z_0^i = \tilde{z}_0^i =0 \in \mathbb{R}^n$, $\mu_\alpha = \sum_{i=1}^{m}\mu_i$, $\alpha = \mu_\alpha / \lambda_{\min}(W)$, $L_\alpha = L + \alpha \lambda_{\max}(W)$,  $\tilde q = \frac{\mu_\alpha}{L_\alpha}$, $q =\tilde q \frac{\lambda^{+}_{\min}(W)}{\lambda_{\max}(W)}$, $\alpha_0$ solves $\frac{\alpha_0^2 -q}{1-\alpha_0} = 1$, $\tilde \alpha_0$ solves $\frac{\tilde\alpha_0^2 -\tilde q}{1-\tilde\alpha_0} = 1$, $T$ and $N$.}
		\STATE{For each agent $i$}
		\FOR{ $k=0,1,2,\cdots,N-1$ }
		\STATE{ $w^i_0 = \tilde{w}^i_0 = 0 \in \mathbb{R}^n$}
		\FOR{ $t=0,1,2,\cdots,T-1$}
				\STATE{Share $\tilde w^i_t$ with neighbors, i.e. $\{j \mid (i,j) \in E \}$.}
		\STATE{$w^i_{t+1} = \tilde{w}^i_t + \frac{1}{L_\alpha} (\tilde z^i_k -\nabla f_i (\tilde{w}^i_t) - \alpha \sum_{j=1}^{m} W_{ij} \tilde{w}^j_t ) $}
		\STATE{Compute $\tilde \alpha_{t+1}\in (0,1)$ from $\tilde \alpha_{t+1}^2 = (1 -  \tilde\alpha_{t+1}) \tilde \alpha_{t}^2 + \tilde q \tilde \alpha_{t+1}$ and set $\tilde \beta_t = \frac{\tilde \alpha_t(1-\tilde\alpha_t)}{\tilde \alpha_{t}^2 + \tilde  \alpha_{t+1}}$}
		\STATE{{ $\tilde w_{t+1}^i  = w_{t+1}^i + \tilde \beta_t(w_{t+1}^i - w_t^i)$}}
		\ENDFOR
		\STATE{Share $w^i_T$ with neighbors, i.e. $\{j \mid (i,j) \in E \}$.}
		\STATE{{ $z_{k+1}^i = \tilde z_k^i - \frac{\mu_\alpha}{\lambda_{\max}(W)} \sum_{j=1}^{m} W_{ij} w^j_T$}}
		\STATE{Compute $\alpha_{k+1}\in (0,1)$ from $\alpha_{k+1}^2 = (1 - \alpha_{k+1})\alpha_{k}^2 +q \alpha_{k+1}$ and set $\beta_k = \frac{\alpha_k(1-\alpha_k)}{\alpha_{k}^2 + \alpha_{k+1}}$}
		\STATE{{ $\tilde z_{k+1}^i  = z_{k+1}^i + \beta_k(z_{k+1}^i - z_k^i)$}}
		\ENDFOR
	\end{algorithmic}
\label{alg:improve1}
\end{algorithm}

\begin{corollary}\label{cor:improve_bound}
	Let $F(x)$ be defined in~\eqref{equiv_main_problem}, and assume $f_i$ is $L_i$-smooth for $1\leq i \leq m$ and $\bar \mu = \sum_{i=1}^{m}\mu_i >0$. For any $\varepsilon >0$, the output $x^*(z_N)$ of Algorithm~\ref{alg:improve1} is an $(\varepsilon,\varepsilon/R )$-solution of~\eqref{consensus_problem2} for 
	\begin{align*}
	N & \geq 8\sqrt{\left( \frac{L}{\bar \mu} + \chi(W)\right)\chi(W)}\log \left(\frac{2\sqrt{2}\lambda_{\max}(W) R^2}{\bar \mu\cdot\varepsilon}\right),
	\end{align*}
	and
	\begin{align*}
	T & \geq \sqrt{\frac{L}{\bar \mu} + \chi(W)} \log \left(\frac{6L_\alpha R^2 R_w^2}{\varepsilon^2}\sqrt{\left( \frac{L}{\bar \mu} + \chi(W)\right) \chi(W)}\right),
	\end{align*}
	where $R=\|y^*\|_2$, $R_x=\|x^*-x^*(0)\|_2 $, $R_w=R_x+\|x^*\|_2$, and $\chi(W)=\lambda_{\max}{(W)}/ \lambda_{\min}^{+}{(W)}$. 
\end{corollary}

Corollary~\ref{cor:improve_bound} implies that at each of the outer iterations, required to obtain an approximate solution to the inner maximization problem, the number of oracle calls for $f$ and communication rounds between agents can be bounded by 
\begin{align*}
	\tilde O \left(  \sqrt{\frac{L_\alpha}{\mu_\alpha}} \right) = \tilde O \left(  \sqrt{\frac{L}{\bar \mu} + \chi(W)} \right).
\end{align*} 
Moreover, the number of outer communication rounds can be bounded by
\begin{align*}
\tilde{O}\left( \sqrt{\frac{L_\alpha}{\mu_\alpha} \chi(W)}\right) = \tilde O \left(  \sqrt{\left( \frac{L}{\bar \mu} + \chi(W)\right) \chi(W)} \right).
\end{align*}

The total number of communications rounds and local oracle calls, taking into account the inner and outer loops is $\tilde{O}\big(\left(\cub{{L}/{\bar \mu}} + \chi(W)\right) \sqrt{\chi(W)} \big)$.

This estimate shows that we can replace the smallest strong convexity constant for the sum among all of them, but we have to pay an additive price proportional to the condition number of the graph and additional communication rounds in the inner maximization problem proportional to the number of oracle calls for $f$. 

This result can be extended to the case when $F(x)$ is just smooth by using the regularization technique with $\mu_i = \varepsilon / (R^2_x)$. Particularly, consider the regularized function
\begin{align}\label{better_cond2}
\hat F_\alpha (x) & = F(x) +  \frac{\varepsilon}{R_x^2}\|x - x^*(0)\|_2^2 + \frac{\alpha}{2}\left\langle x,Wx\right\rangle \nonumber \\ 
&= \sum\limits_{i=1}^{m}f_i(x_i) +  \frac{\varepsilon}{R_x^2}\|x - x^*(0)\|_2^2 + \frac{\alpha}{2}\left\langle x,Wx\right\rangle.
\end{align}

The function $\hat F_\alpha$ is \mbox{$\hat \mu_\alpha$-strongly} convex with $\hat \mu_\alpha = \min \left\lbrace m\frac{\varepsilon}{R^2_x},\alpha \lambda_{\min}(W)\right\rbrace $ and has \mbox{$\hat L_\alpha $-Lipschitz} continuous gradients with $\hat L_\alpha = L + \alpha \lambda_{\max}(W) +  m{\varepsilon}/{R^2_x}$. Choose $\alpha = {m({\varepsilon}/{R^2_x} )}/{\lambda_{\min}(W)}$. Under this specific choice of $\alpha$, the function $F_\alpha$ will have a condition number
\begin{align*}
\frac{\hat L_\alpha}{\hat \mu_\alpha} = \frac{L}{ m\frac{\varepsilon}{R^2_x}} + \frac{\lambda_{\max}(W)}{\lambda_{\min}(W)} +1 = \frac{{R^2_x} L}{ m{\varepsilon}} + \chi(W) +1.
\end{align*}

The next Corollary shows the complexity of the proposed distributed augmented algorithm for the solution of sums of smooth convex functions.


\begin{algorithm}[t]
	\begin{algorithmic}[1]
		\STATE{All agents set $z_0^i = \tilde{z}_0^i =0 \in \mathbb{R}^n$, $\hat \mu_\alpha = m\frac{\varepsilon}{R^2_x}$, $\alpha = \hat \mu_\alpha / \lambda_{\min}(W)$, $\hat L_\alpha = L + \alpha \lambda_{\max}(W) +  m\frac{\varepsilon}{R^2_x}$,  $\tilde q = \frac{\hat \mu_\alpha}{\hat L_\alpha}$, $q =\tilde q \frac{\lambda^{+}_{\min}(W)}{\lambda_{\max}(W)}$, $\alpha_0$ solves $\frac{\alpha_0^2 -q}{1-\alpha_0} = 1$, $\tilde \alpha_0$ solves $\frac{\tilde\alpha_0^2 -\tilde q}{1-\tilde\alpha_0} = 1$, $T$ and $N$.}
		\STATE{For each agent $i$}
		\FOR{ $k=0,1,2,\cdots,N-1$ }
		\STATE{ $w^i_0 = \tilde{w}^i_0 = 0 \in \mathbb{R}^n$}
		\FOR{ $t=0,1,2,\cdots,T-1$}
		\STATE{Share $\tilde w^i_t$ with neighbors, i.e. $\{j \mid (i,j) \in E \}$.}
		\STATE{$w^i_{t+1} = \tilde{w}^i_t + \frac{1}{L_\alpha} (\tilde z^i_k -\nabla f_i (\tilde{w}^i_t) - \alpha \sum_{j=1}^{m} W_{ij} \tilde{w}^j_t - \frac{\varepsilon}{R^2_x}\left( w_i - x^*_i(0)\right) ) $}
		\STATE{Compute $\tilde \alpha_{t+1}\in (0,1)$ from $\tilde \alpha_{t+1}^2 = (1 -  \tilde\alpha_{t+1}) \tilde \alpha_{t}^2 + \tilde q \tilde \alpha_{t+1}$ and set $\tilde \beta_t = \frac{\tilde \alpha_t(1-\tilde\alpha_t)}{\tilde \alpha_{t}^2 + \tilde  \alpha_{t+1}}$}
		\STATE{{ $\tilde w_{t+1}^i  = w_{t+1}^i + \tilde \beta_t(w_{t+1}^i - w_t^i)$}}
		\ENDFOR
		\STATE{Share $w^i_T$ with neighbors, i.e. $\{j \mid (i,j) \in E \}$.}
		\STATE{{ $z_{k+1}^i = \tilde z_k^i - \frac{\hat \mu_\alpha}{\lambda_{\max}(W)} \sum_{j=1}^{m} W_{ij} w^j_T$}}
		\STATE{Compute $\alpha_{k+1}\in (0,1)$ from $\alpha_{k+1}^2 = (1 - \alpha_{k+1})\alpha_{k}^2 +q \alpha_{k+1}$ and set $\beta_k = \frac{\alpha_k(1-\alpha_k)}{\alpha_{k}^2 + \alpha_{k+1}}$}
		\STATE{{ $\tilde z_{k+1}^i  = z_{k+1}^i + \beta_k(z_{k+1}^i - z_k^i)$}}
		\ENDFOR
	\end{algorithmic}
	\caption{Augmented Distributed \textsc{FGM} for smooth problems}
	\label{alg:improve2}
\end{algorithm}

\begin{corollary}\label{cor:improve_bound2}
	Let $F(x)$ be a function such that \cub{Assumption~\ref{assumptions}(c)} hold. Then, for any $\varepsilon >0$, the output $x^*(z_N)$ of Algorithm~\ref{alg:improve1} is an $(\varepsilon,\varepsilon/R )$-solution of~\eqref{consensus_problem2} for 
	\begin{align*}
	N & \geq 8\sqrt{\left( \frac{2R^2_x L}{ m{\varepsilon}} + \chi(W) +1\right) \chi(W)}\log C_1,
	\end{align*}
	and
	\begin{align*}
	T & \geq \sqrt{\frac{2R^2_x L}{ m{\varepsilon}} + \chi(W) +1} \log C_2,
	\end{align*}
	where
	\begin{align*}
	C_1 & = \frac{8\sqrt{2}\lambda_{\max}(W) R^2_x R^2}{m\cdot\varepsilon^2}\\
	C_2 & = \frac{24 (L+ \alpha \lambda_{\max}(W) +  m\frac{\varepsilon}{R^2_x}) R^2 R_w^2}{\varepsilon^2}\sqrt{\left( \frac{2R^2_x L}{ m{\varepsilon}} + \chi(W) +1\right) \chi(W)},
	\end{align*}
	$R = \|y^*\|_2$, $R_x = \|x^*-x^*(0)\|_2 $, $R_w = R_x + \|x^*\|_2$, and $\chi(W)=~ {\lambda_{\max}{(W)}}/{\lambda_{\min}^{+}{(W)}}$. 
\end{corollary}

The number of inner communication rounds and local oracle calls required by Algorithm~\ref{alg:improve2} to obtain an $(\varepsilon,\varepsilon/R)$-solution of~\eqref{consensus_problem2} can be bounded by $\tilde{O}\big(\sqrt{{LR_x^2}/{(m\varepsilon)} + \chi(W)}\big)$. On the other hand the number of outer communication rounds can be bounded by \cub{$\tilde{O}\big(\sqrt{\left({LR_x^2}/{(m\varepsilon)}  + \chi(W)\right)\chi(W)}\big)$}. Therefore, this approach is useful for large but well-connected networks, where $m \gg 1$ and $\chi(W) = O(1)$ or $\chi(W) = O(\log(m))$.

A similar method to improve the definition of the global strong convexity parameter was proposed in~\cite{sca17}. In~\cite{sca17}, the authors propose to introduce the proxy function $f_i(x)- (\mu_i - ({1}/{m})\sum_{i=1}^{m}\mu_i)\|x\|_2^2$. With this new function, the condition number of $F$ improves to
\begin{align*}
\frac{\max_i L_i -\mu_i}{\frac{1}{m}\sum_{i=1}^{m}\mu_i} -1.
\end{align*}
	
	\section{Discussion}\label{sec:discussion}
	
	Table \ref{tab:summary} presents a summary of the results presented in Section~\ref{sec:main}. In particular, it shows the number of communication rounds required to obtain an $(\varepsilon,\varepsilon/R)$-solution for each function class in Assumption~\ref{assumptions}. 

\begin{table}
	\centering
	\tbl{A summary of algorithmic performance.}
		{\begin{tabular}{cc} \toprule
			\bf Property of $F(x)$ & \bf Iterations Required  \\ \toprule
			$\mu$-strongly convex and $L$-smooth & ${O}\left( \sqrt{\frac{L}{\mu} \chi(W)}  \cu{\log\left(\frac{\lambda_{\max}(W)R^2}{\mu \varepsilon}\right)}\right)$          \\
			$\mu$-strongly convex and $M$-Lipschitz&${O}\left( \sqrt{\frac{M^2}{\mu \varepsilon} \chi(W)}\cu{\log\left(\frac{\chi(W)M^2}{\mu \varepsilon}\right)}\right) $ \\
			$L$-smooth &   ${O}\left( \sqrt{\frac{L R^2_x}{\varepsilon} \chi(W)}\cu{\log \left(\frac{\sqrt{\lambda_{\max}(W)}RR_x}{\varepsilon}\right)}  \right) $ \\
			$M$-Lipschitz& ${O}\left(  \sqrt{\frac{M^2R_x^2}{\varepsilon^2} \chi(W)} \cu{\log\left(\frac{\sqrt{\chi(W)}MR_x}{\varepsilon}\right)}\right) $\\ \botrule
		\end{tabular}}
	\label{tab:summary}    
\end{table}

The estimates in Table \ref{tab:summary} are optimal up to logarithmic factors. In the smooth cases, where $L < \infty$, these estimates follow from classical centralized complexity estimation of the \textsc{FGM} algorithm. In the distributed setting, one has to perform $O(\sqrt{\chi(W)} \log (1/\varepsilon))$ additional consensus steps at each iteration. This corresponds to the number of iterations needed to solve the consensus problem 
\begin{align}\label{prob:consensus}
\min_{{x}} \frac{1}{2} \left\langle {x,Wx} \right\rangle,
\end{align}
where $W$ is a communication matrix as defined in Section~\ref{sec:problem}. \textsc{FGM} provides a direct estimate on the number of iterations required to reach consensus, given that \eqref{prob:consensus} is $\sigma_{\min}(\sqrt{W})$-strongly convex in $x_0 +\ker(W)$ and has $\sigma_{\max}(\sqrt{W})$-Lipschitz continuous gradients, and this estimate cannot be improved up to constant factors. 

The specific value of $\chi(W)$, and its dependency on the number of nodes $m$ has been extensively studied in the literature of distributed optimization \cite{ned15}. In \cite{Nedic2017}, Proposition $5$ provides an extensive list of \textit{worst-case} dependencies of the spectral gap for large classes of graphs. Particularly, for fixed undirected graphs, in the worst case we have $\chi(W) = O(m^2)$ \cite{ols14}. This matches the best upper bound found in the literature of consensus and distributed optimization \cite{ore10,liu13}. Thus, the consensus set described by the constraint ~\mbox{$\sqrt{W}x=0$} should be preferred over the description as $Wx=0$, even though both representations correctly describe the consensus subspace $x_1=\hdots=x_m$. Particularly, when we pick $A = \sqrt{W}$, we have \mbox{$\chi(A^TA) = \chi(W)$} instead of \mbox{$\chi(W^TW) = \chi(W^2) \gg \chi(W)$}. 

The cases when $F(x)$ is convex or strongly convex can be generalized to $p$-norms, with $p\geq 1$, see \cite{ani17}. The definitions of the condition number $\chi$ needs to be defined accordingly. Let's introduce a norm \mbox{$\| x\|^2_p = \|x_1\|_p^2 + ... + \|x_m\|_p^2$} for \mbox{$p\geq 1$} and assume that $F(x)$ is $\mu$-strongly convex and $L$-Lipschitz continuous gradient in this (new) norm $\|\cdot \|_p$ (in $\mathbb{R}^{mn}$), see \cite{nes15}~(Lemma 1), \cite{dvu17b}~(Lemma 1) and \cite{nes05}~(Theorem 1). Thus, 
\begin{align*}
\chi(W) = {\max_{\|h\|=1} \frac{\langle h,Wh\rangle}{\mu}} \bigg /{\min_{\|h\|=1, h \perp \text{ker}(W)} \frac{\langle h,Wh \rangle}{L}}.
\end{align*} 

Note that we typically do not know $R$ or $R_x$. Thus, we require a method to estimate the strong convexity parameter, which is challenging \cite{nes07,odo15}. Some recent work have explored restarting techniques to reach optimal convergence rates when the strong convexity parameters are unknown \cite{odo15,iou14}. Similarly, a generalization of the \textsc{FGM} algorithm can be proposed when the smoothness parameter is unknown \cite{gas16b}. However, the effect of restarting in the distributed setup requires further study and is out of the scope of this paper. 

\cu{Additionally, parameters such as  $\lambda_{\min}^{+}(W)$, $\lambda_{\max}(W)$ can be efficiently computed in a distributed manner~\cite{TRAN20145526}. Moreover, the values of~$\mu$ and $L$ can be shared with simple max-consensus which is guaranteed to converge in finite time~\cite{maxcon}.}

	\section{Experimental results}\label{sec:experiments}
	
	In this section, we will provide experimental results that show the performance of the optimal distributed algorithm presented in Sections~\ref{sec:main}. We will consider two different graph topologies; the cycle graph and Erd\H{o}s-R\'enyi random graph, of various sizes. We choose the cycle graph ($\chi(W) = O(m^2)$) and the Erd\H{o}s-R\'enyi random graph ($\chi(W) = O(\log (m))$) to show the scalability properties of the algorithms. \cub{Generally, we do not have access to the spectral properties of the graphs. For the cases where the network has a simple structure, like a path graph, star graph, cycle graph, we use its explicit values since we have access to them. For the case of random graphs, we use its lower bounds which are generally found in the literature~\cite[Table 1]{Nedic2019}.}

Initially, consider the \textit{ridge regression }(strongly convex and smooth) problem
\begin{align}\label{eq:regression}
\min_{z \in \mathbb{R}^n} \frac{1}{2ml}\|b - Hz\|_2^2 + \frac{1}{2}c \|z\|_2^2,
\end{align}
to be solved distributedly over a network. Each entry of the data matrix $H \in \mathbb{R}^{ml\times n}$ is generated as an independent identically distributed random variable $H_{ij} \sim \mathcal{N}(0,1)$, the vector of associated values $b \in \mathbb{R}^{ml}$ is generated as a vector of random variables where $b = Hx^* + \epsilon$ for some predefined $x^* \in \mathbb{R}^{n}$ and $\epsilon \sim \mathcal{N}(0,0.1)$. The columns of the data matrix $H$ and the output vector $b$ are evenly distributed among the agents with a total of $l$ data points per agent. The regularization constant is set to $c = 0.1$. 

Figure \ref{fig:ex1_optimality} shows experimental results for the ridge regression problem for a cycle graph and an Erd\H{o}s-R\'enyi random graph. For each type of graph we show the distance to optimality as well as the distance to consensus for a fixed graph with $m=100$, $n=10$ and $l=100$. Additionally, the scalability of the algorithm is shown by plotting the required number of steps to reach an accuracy of $\epsilon = 1\cdot 10^{-10}$ versus the number of nodes in the graph. We compare the performance of the proposed algorithm with some of the state of the art methods for distributed optimization. \textsc{{Dist-Opt}} refers to  Algorithm~\ref{alg:case1}. \textsc{{NonAcc-Dist}} refers to the non-accelerated version of Algorithm~\ref{alg:case1}. \textsc{{FGM}} is the centralized \textsc{FGM}. \textsc{Acc-DNGD} refers to the algorithm proposed in \cite{Qu2017} with parameter \mbox{$\eta = 0.1$} and $\alpha = \sqrt{\mu \eta}$. \textsc{{EXTRA}} refers to the algorithm proposed in \cite{shi15} with parameter $\alpha = 1$. \textsc{{DIGing}} refers to the algorithm proposed in \cite{ned16w} with parameter $\alpha = 0.1$. Figure \ref{fig:ex1_optimality} shows linear convergence rate with faster performance than other algorithms and linear scalability with respect to the size of the cycle graphs. \cu{When available, we have used the suggested parameter selection for each of the algorithms compared.}

\begin{figure}[t]
	\centering
	\resizebox{\textwidth}{!}{
\includegraphics[width=0.21\textwidth] {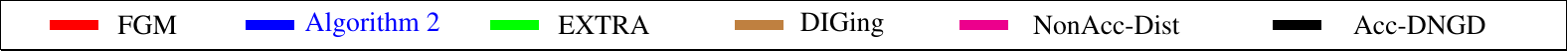}	
} \\
	\subfigure[Cycle Graph]{\resizebox{0.26\textwidth}{!}{
		\includegraphics[width=0.21\textwidth] {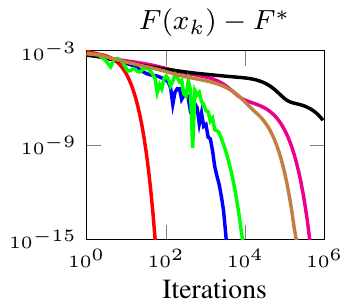}
	}
		\resizebox{0.24\textwidth}{!}{
	\includegraphics[width=0.21\textwidth] {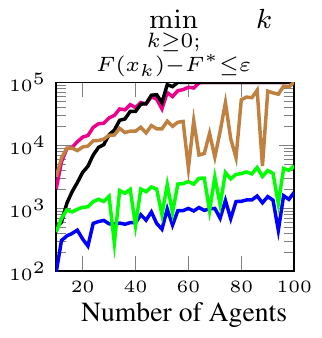}	
	}
	\resizebox{0.26\textwidth}{!}{
\includegraphics[width=0.21\textwidth] {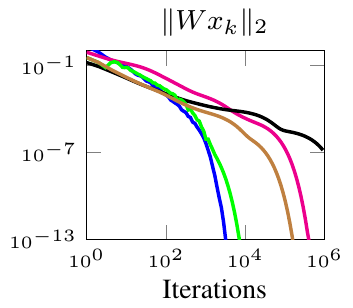}	
}
\resizebox{0.24\textwidth}{!}{
	\includegraphics[width=0.21\textwidth] {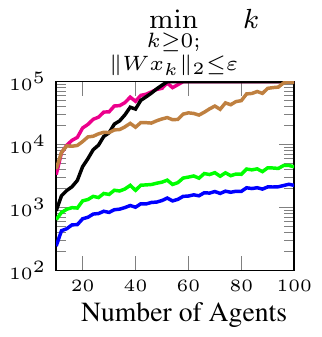}
}
}\\
	\subfigure[Erd\H{o}s-R\'enyi random graphs.]{
			\resizebox{0.26\textwidth}{!}{
				\includegraphics[width=0.21\textwidth] {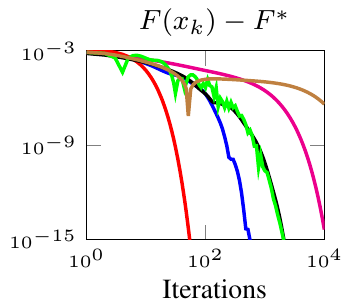}
			}
			\resizebox{0.24\textwidth}{!}{
				\includegraphics[width=0.21\textwidth] {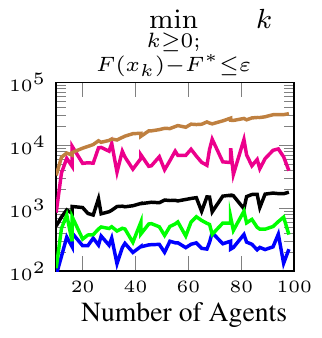}	
			}
			\resizebox{0.26\textwidth}{!}{
				\includegraphics[width=0.21\textwidth] {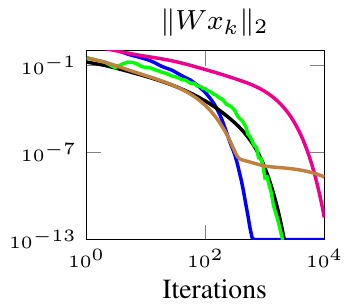}	
			}
			\resizebox{0.24\textwidth}{!}{
				\includegraphics[width=0.21\textwidth] {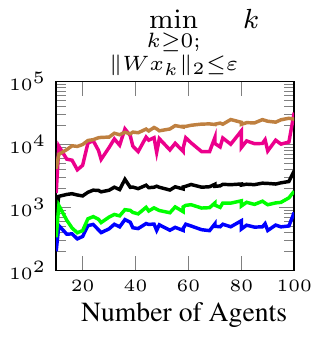}
			}
}
	\\
	\caption{Distance to optimality and consensus, and network scalability for a strongly convex and smooth problem. (a) Results for a cycle graph.  (b) Results for an Erd\H{o}s-R\'enyi random graph.}
	\label{fig:ex1_optimality}
\end{figure} 

As a second example, consider the Kullback-Leibler~(KL) barycenter computation problem (strongly convex and $M$-Lipschitz)
	\begin{align*}
	\min_{z \in S_n(1)} \sum\limits_{i=1}^{m}D_{KL}(z\|q_i) \triangleq \sum\limits_{i=1}^{m}\sum\limits_{j=1}^{n}z_i \log\left({z_i}/{[q_i]_j} \right),  
	\end{align*}
where $S_n(1) = \{ z \in \mathbb{R}^n : z_j \geq 0 ; j=1,2,\hdots,n; \sum_{j=1}^{n}z_j = 1 \}$ is a unit simplex in $\mathbb{R}^n$ and $q_i \in S_n(1)$ for all $i$. Each agent has a private probability distribution $q^i$ and seek to compute the a probability distribution that minimizes the average KL distance to the distributions $\{q_i\}_{i=1,\hdots,m}$. Figure \ref{fig:ex2_optimality} shows the results for the KL barycenter problem for a cycle graph with $m=100$, $n=10$ and various values of the regularization parameter when Algorithm~\ref{alg:case2} is used. We show the distance to optimality as well as the distance to consensus and the scalability of the algorithm.

\begin{figure*}[tbp!]
	\centering
			\resizebox{0.35\textwidth}{!}{
		\includegraphics[width=0.21\textwidth] {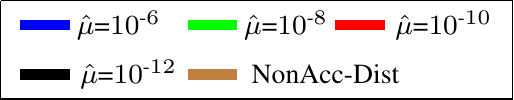}	
		} \hspace{2cm}
		\resizebox{0.175\textwidth}{!}{
	\includegraphics[width=0.21\textwidth] {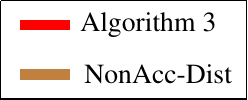}	
	} \\
\resizebox{0.51\textwidth}{!}{	\subfigure{
		\includegraphics[width=0.21\textwidth] {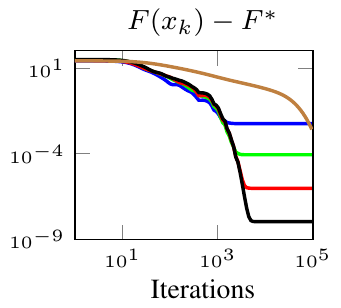}
		\includegraphics[width=0.21\textwidth] {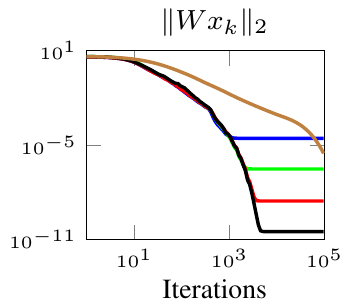}
	} }
	\resizebox{0.48\textwidth}{!}{\subfigure{
				\includegraphics[width=0.21\textwidth] {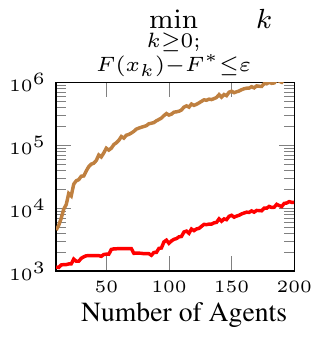}
		\includegraphics[width=0.21\textwidth] {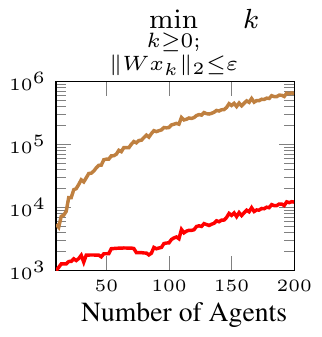}	
	} }
	\caption{Distance to optimality and consensus, and network scalability for a strongly convex and $M$-Lipschitz problem over a cycle graph with $m=100$, $n=10$ and various values of the regularization parameter $\hat \mu$ for Algorithm~\ref{alg:case2}. The brown line shows the performance for the non-accelerated distributed gradient descent of the dual problem.}
	\label{fig:ex2_optimality}
\end{figure*} 

In~\eqref{eq:regression}, if we assume $c=0$ and $H_i$ is a \textit{wide} matrix where $n \gg l$ (i.e., the dimension of the data points is much larger than the number of data points per agent), then the resulting problem is smooth but no longer strongly convex. Figure \ref{fig:ex3_optimality} shows the performance of Algorithm~\ref{alg:case3} over a cycle graph and an Erd\H{o}s-R\'enyi random graph, where $m=50$, $n=20$ and $l=10$, for different values of the regularization parameter. As expected, smaller values of the regularization parameter increase the precision of the algorithm but hinder its convergence rate. We compare the performance of Algorithm~\ref{alg:case3} with the distributed accelerated method proposed in~\cite{Qu2017} for non-strongly convex functions (Acc-DNGD-NSC) for a fixed regularization value $\hat \mu = 1\cdot 10^{-6}$. As presented in Table \ref{tab:summary_other}, the algorithms have similar convergence rates, as shown by the intersection of the curves around the accuracy point corresponding to the regularization parameter. Nevertheless, as seen in Figure \ref{fig:ex1_optimality}, Acc-DNDG-NSC has a worst scalability with respect to the number of nodes, which is particularly evident for the cycle graph.

\begin{figure}[tbp!]
	\centering
	\resizebox{\textwidth}{!}{
			\includegraphics[width=0.21\textwidth] {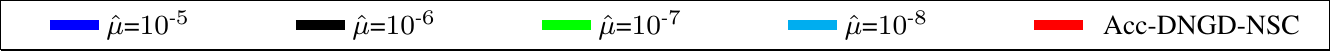}
	} \\
	\subfigure[Cycle graph]{  \resizebox{0.6\textwidth}{!}{
			\includegraphics[width=0.21\textwidth] {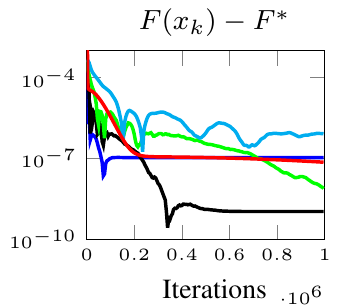}
\includegraphics[width=0.21\textwidth] {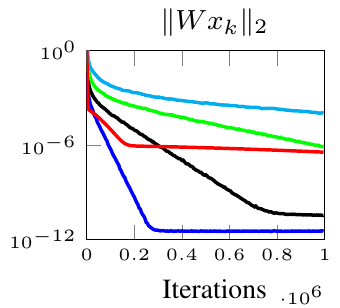}
}
}
	\subfigure[Erd\H{o}s-R\'enyi random graph]{  \resizebox{0.6\textwidth}{!}{ 
	\includegraphics[width=0.21\textwidth] {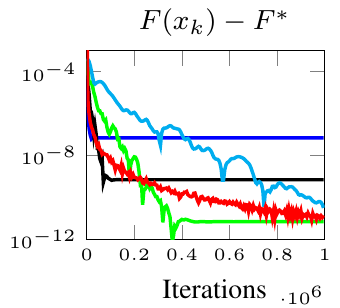}	
\includegraphics[width=0.21\textwidth] {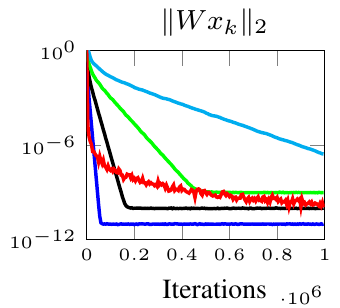}
}
}
	\caption{Distance to optimality and consensus for a smooth problem over a Erd\H{o}s-R\'enyi random graph with $m=100$, $n=50$, $l=10$ and various values of the regularization parameter $\varepsilon$ for Algorithm~\ref{alg:case3}.}
	\label{fig:ex3_optimality}
\end{figure}

	Next, we present a number of numerical experiments and compare the performance of the proposed methods. Consider the logistic regression problem for training linear classifiers. We seek to solve the following optimization problem:
\begin{align}\label{fun:logistic}
\min_{x\in \mathbb{R}^n} \frac{1}{2ml} \sum_{i=1}^{ml} \log \left(1 + \exp\left(-y_i \cdot A_i^Tx\right)\right) + \frac{1}{2}c \|x\|_2^2,
\end{align}
where $A_i \in \mathbb{R}^d$ is a data point with $y_i \in \{-1,1\}$ as its corresponding class assignment. We assume there is a total of $ml$ data points distributed evenly among $m$ agents, where each agent holds $l$ data points. For our experiments, initially we generate a random vector $x_{\text{true}} \in \mathbb{R}^n$ where each entry is chosen uniformly at random on~$[-1,1]$, we fixed $c=0.1$, the data points $A_i$ are generated uniformly at random on~$[-1,1]^n$, and \cub{each} label is computed as $y_i = \text{sign}( A_i^T x_{\text{true}})$. Note that each of the agents in the network will have a local function
\begin{align}\label{logic_local}
f_i(x) & = \frac{1}{2ml} \sum_{j=1}^{l} \log \left(1 + \exp\left(-[y^i]_j \cdot [A^i]_j^Tx\right)\right) + \frac{1}{2m}c \|x\|_2^2,
\end{align} 
where $A^j \in \mathbb{R}^{l\times n }$ and $y^j \in \{-1,1\}^l$ are the data points held by agent $j$ and their corresponding class assignments. Moreover, \eqref{logic_local} is not dual friendly. Therefore, we will use Algorithm~\ref{alg:no_friend:case1} for our next set of experimental results.

Figure~\ref{fig:logistic} shows the distance to optimality and the distance to consensus of the output of Algorithm~\ref{alg:no_friend:case1} for the problem of logistic regression. We use cycle graphs and Erd\H{o}s-R\'enyi random graph for a problem with $10000$ data points of dimension $10$. For each class of graphs, we explore three different scenarios for the distribution of the data among agents. We present the results for networks of $10$, $100$, and $1000$ agents; where each agent holds $1000$, $100$ and $10$ data points respectively. We compare the results of the \textsc{Acc-DNGD} algorithm in~\cite{Qu2017} with parameter \mbox{$\eta = 0.1$} and $\alpha = \sqrt{\mu \eta}$, the \textsc{{EXTRA}} algorithm in~\cite{shi15} with parameter $\alpha = 1$, and the \textsc{{DIGing}} algorithm in~\cite{ned16w} with parameter $\alpha = 0.1$. Figure~\ref{fig:logistic} shows a faster geometric convergence rate of Algorithm~\ref{alg:no_friend:case1} with respect to \textsc{Acc-DNGD}, \textsc{{EXTRA}} and \textsc{{DIGing}}. Nonetheless, we point out that those algorithms could be subject to improved convergence rates if the particular parameters of each algorithm are carefully selected. In the presented results, we do not claim to have selected the optimal step sizes for the algorithms we are comparing our proposed method. For the cycle graph in Figure~\ref{fig:logistic}(a), as the size of the network increases and the number of points per agent decreases the convergence rates slows down. The  \textsc{{EXTRA}} algorithms seem to have a near-optimal scaling on its convergence rate with respect to the size of the network. The  \textsc{Acc-DNGD} and \textsc{{DIGing}} algorithms rapidly decrease their convergence rate with the size of the network. Due to the better condition number of the Erd\H{o}s-R\'enyi random graphs, the Figure~\ref{fig:logistic}(b) shows a better scaling with the size of the network for all the analyzes algorithms. For this class of networks, the \textsc{Acc-DNGD} algorithms outperforms \textsc{{EXTRA}} and \textsc{{DIGing}}.

\begin{figure}[tbp!]
	\centering
	\resizebox{0.7\textwidth}{!}{
		\includegraphics[width=\textwidth] {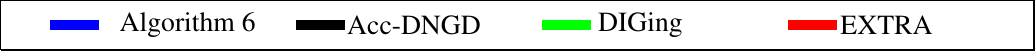}
	} \\
	\subfigure{
					\includegraphics[width=0.34\textwidth] {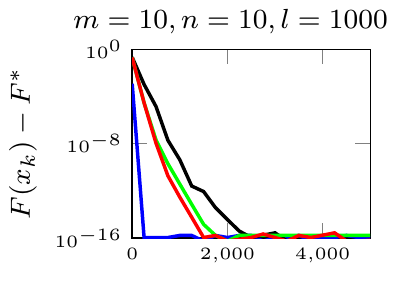}
					\includegraphics[width=0.30\textwidth] {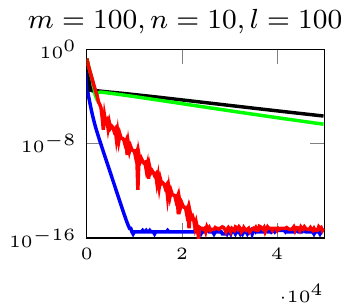}
					\includegraphics[width=0.30\textwidth] {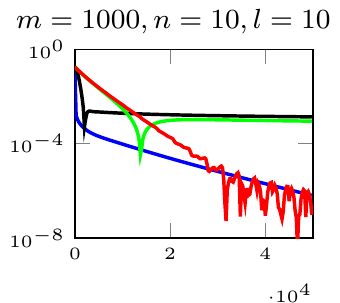}
	}\\
	\subfigure[Cycle Graph]{\addtocounter{subfigure}{-1}
		\includegraphics[width=0.34\textwidth] {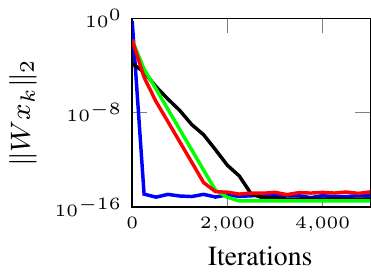}
		\includegraphics[width=0.30\textwidth] {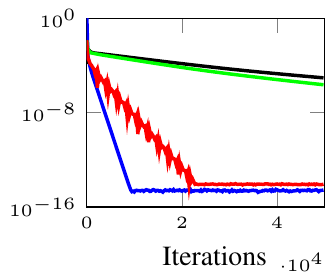}
		\includegraphics[width=0.30\textwidth] {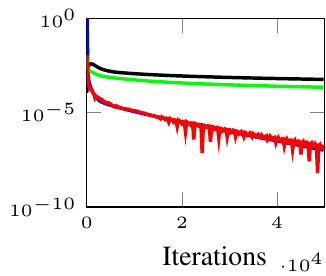}
}
\\ \vspace{-0.4cm}
			\subfigure{
	\includegraphics[width=0.34\textwidth] {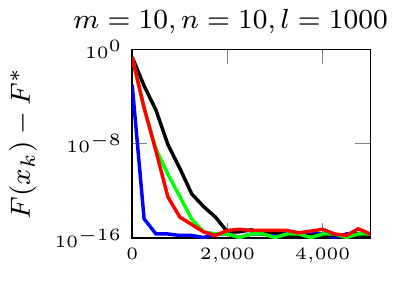}
	\includegraphics[width=0.30\textwidth] {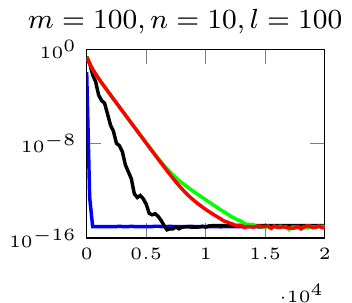}
	\includegraphics[width=0.30\textwidth] {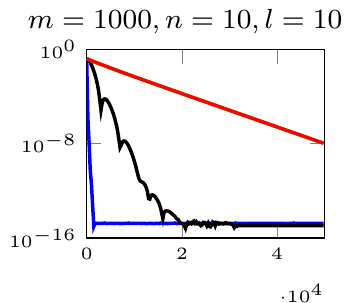}
}\\
\subfigure[Erd\H{o}s-R\'enyi Random Graph]{\addtocounter{subfigure}{-1}
	\includegraphics[width=0.34\textwidth] {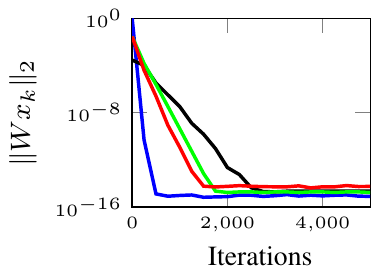}
	\includegraphics[width=0.30\textwidth] {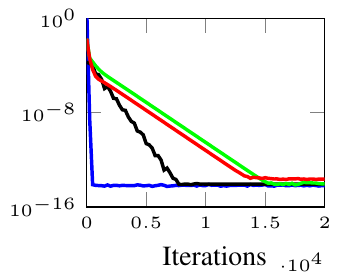}
	\includegraphics[width=0.30\textwidth] {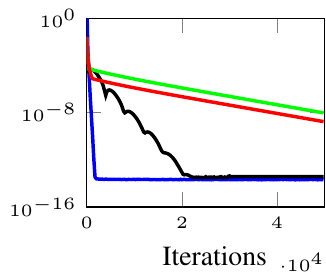}
}
	\caption{Logistic regression on synthetic data over a cycle graph and an Erd\H{o}s-R\'enyi random graph for a total of $10000$ data points and various graph sizes evenly distributing the data points among the agents.}
	\label{fig:logistic}
\end{figure} 

In Figure~\ref{fig:logistic_real}, we use datasets from the library \textsc{LibSVM}~\cite{Chang2011} to compare the performance of Algorithm~\ref{alg:no_friend:case1} as in Figure~\ref{fig:logistic}. We seek to distributedly solve the logistic regression problem over the following datasets: \textsc{a9a}, \textsc{mushrooms}, \textsc{ijcnn1} and \textsc{phishing}. Table~\ref{tab:real_Data} shoes a brief description of the four datasets used. For each problem, we created an Erd\H{o}s-R\'enyi random graph with $100$ agents and evenly distributed the data points among all agents. Algorithm~\ref{alg:no_friend:case1} outperforms the other compared algorithms where the \textsc{Acc-DNGD} having the second best performance following the same scaling patterns as in Figure~\ref{fig:logistic} for Erd\H{o}s-R\'enyi random graphs. The \textsc{{EXTRA}} and \textsc{{DIGing}} algorithms have a worst scaling of their convergence rate as the size of the network increases. 

\begin{table}
		\centering
	\tbl{\textbf{Real Datasets from the \textsc{LibSVM} Library}.}
	{
		\begin{tabular}{cccc} \toprule
			\bf Name & \bf Classes & \bf Data points & \bf Features    \\ \toprule
					\textsc{a9a}  & 2 & 32561 & 123 \\
				 	\textsc{mushrooms} & 2 & 8214 & 112\\
				 	\textsc{ijcnn1}  &	2  & 49990 & 22\\
				 	\textsc{phishing}  & 2& 11055 & 68\\ \botrule
		\end{tabular}
	}
	\label{tab:real_Data}    
\end{table}

\begin{figure}[tbp!]
	\centering
	\includegraphics[width=0.7\textwidth] {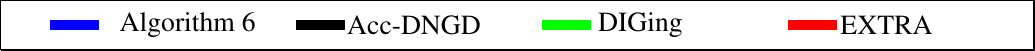}	
 \\
	\subfigure[Distance to Optimality]{\resizebox{\textwidth}{!}{
			\includegraphics[width=0.30\textwidth] {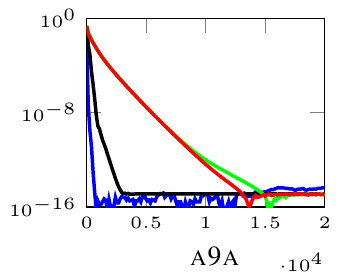}
			\includegraphics[width=0.30\textwidth] {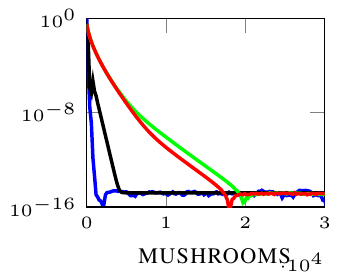}
			\includegraphics[width=0.30\textwidth] {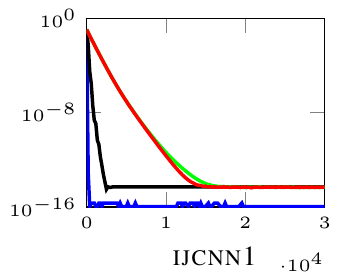}
			\includegraphics[width=0.30\textwidth] {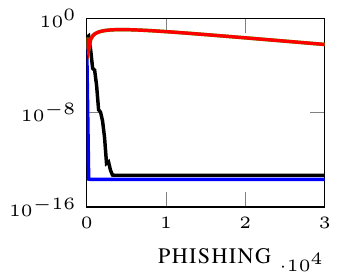}
			}
	}\\
	\subfigure[Distance to Consensus]{\resizebox{\textwidth}{!}{
			\includegraphics[width=0.30\textwidth] {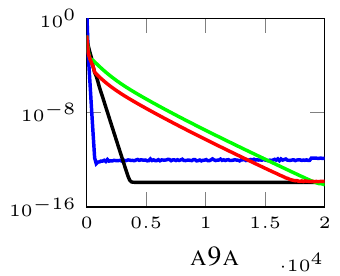}
			\includegraphics[width=0.30\textwidth] {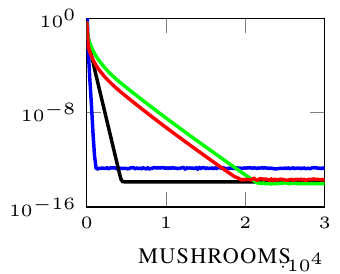}
			\includegraphics[width=0.30\textwidth] {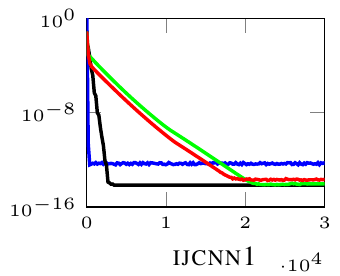}
			\includegraphics[width=0.30\textwidth] {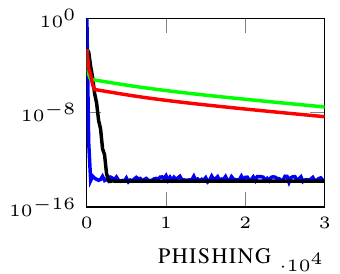}
	}
	}
	\caption{Logistic regression results with data from the \textsc{LibSVM} Library on an Erd\H{o}s-R\'enyi random graph with $100$ agents. }
	\label{fig:logistic_real}
\end{figure}

	\section{Conclusions}\label{sec:conclusions}

We have provided convergence rate estimates for the solution of convex optimization problems in a distributed manner for dual-friendly functions. The provided complexity bounds depend explicitly on the properties of the function to be optimized. If $F(x)$ is smooth, then our estimates are optimal up to logarithmic factors otherwise our estimates are optimal up to constant factors (in terms of t the number of communication rounds). The inclusion of the graph properties in terms of $\sqrt{\chi(W)}$ shows the additional price to be paid in contrast with classical (centralized/non-distributed) optimal estimates. The authors recognize that the proposed algorithms required, to some extent, global knowledge about the graph properties and the condition number of the network function. Nevertheless, our aim was to provide a theoretical foundation for the performance limits of the distributed algorithms. The cases where global information is not available require additional study. 

The proposed algorithms and their corresponding analysis use Nesterov's smoothing to induce strong-convexity either in the primal or dual problems. However, as noted in~\cite{tra15}, Nesterov's smoothing has some practical limitations, e.g., it requires explicit knowledge of the norms $R$ or $R_x$, and the strong convexity term depends on the desired accuracy $\varepsilon$, which in turn may induce slow convergence rates. The study of distributed adaptive methods~\cite{tra15} requires further work.

One can further extend our results and obtain the same rates of converge when the graphs change with time by using restarting techniques \cite{fer16,ban17}. Nevertheless, we require additional assumptions. Particularly, the network changes should not happen often and nodes must be able to detect when these changes occur. The condition number of the sequence of graphs $\chi(W_k)$ then is the worst one among all the graphs in the execution of the algorithm~\cite{rogozin2018optimal}. Additionally, it is still an open research question whether these optimal convergence rates can be achieved over \textit{directed} networks~\cite{maros2018panda}.
	
	\section{Acknowledgments}
	
	The authors would like to thank the Lund Center for Control of Complex Engineering Systems (LCCC) at Lund University, particularly Anders Rantzer and Pontus Giselsson for organizing the 2017 LCCC Focus Period on Large-Scale and Distributed Optimization from which most of the ideas contained in this paper were initially discussed. \ag{We'd also like to thank Thinh T. Doan, who made a lot very useful comments on the initial version of this text. This comments allows to repairs significant misprints.} 
	
	\section{Funding}
	
	The work of A. Nedi\'c and C.A. Uribe is supported by the National Science Foundation under grant no. CPS 15-44953. The work of A. Gasnikov was supported by RFBR 18-29-03071 mk. \ag{The work of C.A.Uribe, S. Lee, and A. Gasnikov was also partially supported by the Yahoo! Research Faculty Engagement Program.}

	\bibliographystyle{gOMS}
	\bibliography{all_refs3}
	
\end{document}